\documentclass[11pt, draft]{amsart}
\usepackage{amssymb, amstext, amscd, amsmath, amssymb}
\usepackage{mathtools, xypic, paralist, color, dsfont, rotating}
\usepackage{verbatim}
\usepackage{enumerate}
\usepackage{float}

\usepackage{tikz}

\numberwithin{equation}{section}
\usepackage[a4paper]{geometry}
\usepackage{quoting}
\quotingsetup{vskip=.1in}
\quotingsetup{leftmargin=.17in}
\quotingsetup{rightmargin=.17in}

\makeatletter
\def\@cite#1#2{{\m@th\upshape\bfseries%
[{#1\if@tempswa{\m@th\upshape\mdseries, #2}\fi}]}}
\makeatother
\theoremstyle{plain}
\newtheorem{theorem}{Theorem}[section]

\newtheorem{proposition}[theorem]{Proposition}
\newtheorem{lemma}[theorem]{Lemma}
\theoremstyle{definition}
\newtheorem{definition}[theorem]{Definition}

\newtheorem{remark}[theorem]{Remark}

\newtheorem*{acknow}{Acknowledgements}
\theoremstyle{remark}

\mathtoolsset{centercolon}
  \newcommand{\B}{{\mathcal{B}}}

  \newcommand{\F}{{\mathcal{F}}}

  \newcommand{\K}{{\mathcal{K}}}
\renewcommand{\L}{{\mathcal{L}}}
  
  \newcommand{\N}{{\mathcal{N}}}
\renewcommand{\O}{{\mathcal{O}}}
\renewcommand{\P}{{\mathcal{P}}}

  \newcommand{\T}{{\mathcal{T}}}

\def\al{\alpha}
\def\be{\beta}
\def\Ga{\Gamma}

\def\De{\Delta}
\def\de{\delta}

\def\la{\lambda}
\def\La{\Lambda}

\newcommand{\bC}{\mathbb{C}}

\newcommand{\bT}{\mathbb{T}}
\newcommand{\bZ}{\mathbb{Z}}

\newcommand{\Bi}{{\mathbf{i}}}

\newcommand{\foral}{\text{ for all }}
\newcommand{\qand}{\quad\text{and}\quad}

\newcommand{\qfor}{\quad\text{for}\quad}
\newcommand{\qforal}{\quad\text{for all}\ }

\newcommand{\ca}{\mathrm{C}^*}

\newcommand{\ol}{\overline}
\newcommand{\wt}{\widetilde}

\newcommand{\Aut}{\operatorname{Aut}}

\newcommand{\id}{{\operatorname{id}}}

\newcommand{\spn}{\operatorname{span}}

\newcommand{\sumoplus}{\operatornamewithlimits{\sum \strut^\oplus}}
\newcommand{\supp}{\operatorname{supp}}

\newcommand{\sca}[1]{\left\langle#1\right\rangle} 
\newcommand{\bo}[1]{\mathbf{#1}} 
\newcommand{\un}[1]{{\underline{#1}}} 

\addtocontents{toc}{\protect\setcounter{tocdepth}{1}}

\begin{document}

\title[Exactness and nuclearity of Nica-Pimsner algebras]{Finite dimensional approximations for \\ Nica-Pimsner algebras}

\author[E.T.A. Kakariadis]{Evgenios T.A. Kakariadis}
\address{School of Mathematics, Statistics and Physics\\ Newcastle University\\ Newcastle upon Tyne\\ NE1 7RU\\ UK}
\email{evgenios.kakariadis@ncl.ac.uk \hfill{$^\dagger$}}



\begin{abstract}
We give necessary and sufficient conditions for nuclearity of Cuntz-Nica-Pimsner algebras for a variety of quasi-lattice ordered groups.
First we deal with the free abelian lattice case.
We use this as a stepping stone to tackle product systems over quasi-lattices that are controlled by the free abelian lattice and satisfy a minimality property.
Our setting accommodates examples like the Baumslag-Solitar lattice for $n=m>0$ and the right-angled Artin groups.
More generally the class of quasi-lattices for which our results apply is closed under taking semi-direct and graph products.

In the process we accomplish more.
Our arguments tackle Nica-Pimsner algebras that admit a faithful conditional expectation on a small fixed point algebra and a faithful copy of the co-efficient algebra.
This is the case for CNP-relative quotients in-between the Toeplitz-Nica-Pimsner algebra and the Cuntz-Nica-Pimsner algebra.
We complete this study with the relevant results on exactness.
\end{abstract}

\maketitle


\section{Introduction}

A central theme in the C*-theory is the construction of operator algebras from a set of geometric/topological data and the investigation of their properties.
Understanding and expanding the class of nuclear C*-algebras is of particular interest for the classification programme in this respect. 
When the model uses nuclear building blocks and amenable algebraic methods then the C*-output \emph{should} be nuclear.
However nuclearity may be implied by non-nuclear building blocks as well, requiring a sharper analysis.
In this paper we identify necessary and sufficient conditions for nuclearity (and exactness) of a large class of Cuntz-type C*-algebras over quasi-lattices of central interest.
This question has been open for more than 15 years even for the honest free abelian lattice $(\bZ^N, \bZ_+^N)$.

\subsection{Background}

The starting point is Pimsner algebras \cite{Pim97} which quantize a range of $\bZ_+$-transformations through the apparatus of C*-corresponden\-ces.
There are two main classes of C*-algebras one can associate to a C*-correspondence $X$ over $A$: the Toeplitz-Pimsner algebra $\T(X)$ generated by the Fock representation, and the Cuntz-Pimsner algebra $\O(X)$ which is minimal in containing an isometric copy of the transformation.
In his influential work, Katsura \cite{Kat04} gave the correct $*$-relations that define $\O(X)$ and answered a range of questions that are now in the core of the theory.

Arveson \cite{Arv89} introduced the semigroup analogue of C*-correspondences, i.e., that of product systems.
Initially considered to understand semigroup actions of von Neumann algebras, by now they form a topic in its own respect that encompasses a wide range of constructs such as higher-rank graphs and generalized crossed products by (possibly non-invertible) dynamics.
Here one considers a semigroup $\{X_p\}_{p \in P}$ of C*-correspondences along with multiplication rules for a quasi-lattice $(G,P)$.
Fowler \cite{Fow02} realized that the Fock representation attains a universal property as long as the compact operators respect the quasi-lattice structure in the sense of Nica \cite{Nic92}.
That settled the Toeplitz-Nica-Pimnser algebras $\N\T(X)$, yet the Cuntz-analogue remained a mystery for some time.
Motivated by developments in higher-rank graphs, Sims-Yeend \cite{SY11} envisioned the Cuntz-Nica-Pimsner algebra $\N\O(X)$ through the augmented Fock space construction.
Later Carlsen-Larsen-Sims-Vittadelo \cite{CLSV11} resolved the problem of co-universality of $\N\O(X)$ under an injectivity condition and when a co-action is at hand.
Recently it has been shown by Dor-On and Katsoulis \cite{DK18a} that $\N\O(X)$ is minimal also in the sense of Arveson's C*-envelope for abelian lattices.

In the meantime a number of papers imported ideas from groupoids, dynamical systems and iterations to study further the Nica-Pimsner algebras when $(G,P)$ is amenable.
It has become apparent that exactness is more tractable for both $\N\T(X)$ and $\N\O(X)$.
The expectation is that each one is (and hence both are) exact if and only if so is the diagonal $A$.
This has been proven by several authors for general amenable quasi-lattices, e.g. Alabandik-Meyer \cite{AM15}, Rennie-Robertson-Sims \cite{RRS15}, Fletcher \cite{Fle17}.
Nuclearity behaves in a similar manner for $\N\T(X)$, but not for $\N\O(X)$ leaving open the following question:

\vspace{-2pt}
\begin{quoting}
\noindent {\bf Question.} Let $(G,P)$ be a pair that admits a co-action. 
What conditions on $X$ and/or $A$ are equivalent to nuclearity of $\N\O(X)$?
\end{quoting}

\vspace{-2pt}
\noindent
It is necessary to emphasize the importance of amenability here.
In its absence the problem becomes highly intractable.
Even for $A = \bC$, a definitive answer for nuclearity of semigroup algebras was not known until the recent breakthrough work of Li \cite{Li13}.

It has been verified that nuclearity of $A$ implies that of $\N\O(X)$ in several cases, e.g., higher-rank graphs \cite{KP00, KPS15}, C*-dynamics \cite{Kak17}, or in other more general amenable contexts \cite{AM15, RRS15, Seh18}.
However the converse is not true even at the level of $\bZ_+$.
Katsura \cite{Kat04} presents an example of Ozawa for a C*-correspondence whose $\O(X)$ is nuclear but $A$ is not.
Alongside he gave the correct equivalent, namely that the embedding $A \hookrightarrow \O(X)$ be nuclear.
It appears that this is also equivalent for $J_X \hookrightarrow \O(X)$ and $A/J_X$ be nuclear for Katsura's ideal of covariance $J_X$.
Amenability allows to reduce the problem to the fixed point algebra and $J_X$ is exactly the intersection of its positive part with $A$.

\subsection{Main results}

In the current paper we consider the nuclearity question for quasi-lattices that have been at the center of attention in the past years.
First we prove necessary and sufficient conditions for nuclearity for $\bZ_+^N$-product systems.
This is not an immediate consequence of the $\bZ_+$-results, but we need to improve upon Katsura's exhibition on nuclearity and short exact sequences \cite[Appendix]{Kat04}.
We treat this case separately as it allows to highlight our arguments and reduce the complexity of the proof for the general case.
For example one of the key ideas is that elements inside boxes have a common minimal.
Secondly we abstract our methods to accommodate more general quasi-lattices that attain a $\bZ_+^N$-controlled map in the sense of Crisp-Laca \cite{CL07} and inherit the aforementioned $\bZ_+^N$-minimality (see Definition \ref{D:minimality}).
This class includes quasi-lattices such as the Baumslag-Solitar group for $n=m>0$ and the right-angled Artin groups, while we show that it is closed under taking semi-direct products and graph products.

\begin{proof}[\bf Theorem A] 
\textit{(Theorem \ref{T:nuc NO}, Theorem \ref{T:nuc con} and Proposition \ref{P:CNP is in})
Let $\vartheta \colon (G,P) \to (\bZ^N, \bZ_+^N)$ be a controlled map with the minimality property.
Set $I := A \cap B_{(\un{0}, \infty]}$ where $B_{(\un{0}, \infty]}$ is the positive part of the fixed point algebra $\N\O(X)^\be$.
Then the following are equivalent:
\begin{enumerate}
\item
$A / I$ is nuclear and the embedding $I \hookrightarrow {B}_{(\un{0},\infty]}$ is nuclear.
\item
The embedding $A \hookrightarrow \N\O(X)$ is nuclear;
\item
$\N\O(X)$ is nuclear. \qedhere
\end{enumerate}
}
\end{proof}

By a usual gauge-action argument we may replace $\N\O(X)$ with the $\bT^N$-fixed point algebra $\N\O(X)^\be$ in items (ii) and (iii).
Our initial motivation was to explore $\N\O(X)$.
Nevertheless, we are able to tackle more general $\bT^N$-equivariant quotients $\N\P(X)$ of $\N\T(X)$ as long as they admit an isometric copy of $A$ and a map
\begin{equation}
E \colon \N\P(X) \to \ol{\spn}\{t(X_p) t(X_p)^* \mid p \in P\},
\end{equation}
such that $E \otimes_{\max} \id_D$ and $E \otimes \id_D$ are faithful conditional expectations for every C*-algebra $D$.
This assumption is satisfied for the CNP-relative quotients in-between $\N\T(X)$ and $\N\O(X)$ (Proposition \ref{P:CNP is in}).
The nuclearity result is new even for the relative Cuntz-Pimsner algebras of C*-correspondences of Muhly-Solel \cite{MS98}, i.e., when $N=1$.

The ideal $A \cap B_{(\un{0}, \infty]}$ plays a fundamental role in our central theorem and there is a wide class for which we can have a description, namely for \emph{strong compactly aligned product systems over $\bZ_+^N$}. 
Indeed by solving polynomial equations in \cite{DK18b} we have
\begin{equation}
A \cap B_{(\un{0}, \infty]} \subseteq (\bigcap_{i=1}^N \ker \phi_{\Bi})^\perp \cap (\bigcap_{i=1}^{N} \phi_{\Bi}^{-1}(\K X_{\Bi}))
\end{equation}
for this class of $\bZ_+^N$-product systems.
In particular equality holds for $\N\O(X)$.

Finally we show that exactness of $A$ is equivalent to exactness of $\N\T(X)$ and thus of any quotient $\N\P(X)$ that admits an isometric copy of $A$ without any additional hypothesis (Theorem \ref{T:exact} and Theorem \ref{T:ex con}).
Our proof is based on splitting the fixed point algebra.

\subsection*{Organization of the paper}

After giving preliminaries in Section 2, in Section 3 we prove the general C*-results that are required.
In Section 4 we show nuclearity for the $(\bZ^N, \bZ_+^N)$ case.
In Section 5 we use these results to tackle quasi-lattices that attain $\bZ_+^N$-controlled maps.
In Section 6 we give examples of quasi-lattices that fall within our theorem.
In Section 4 and in Section 5 we provide the analogues for exactness for which we do not require $\bZ_+^N$-minimality.

\section{Preliminaries}

\subsection{C*-correspondences}

We need to fix notation for C*-correspondences.
We mainly follow \cite{Kat04, Lan95}.
A \emph{C*-correspondence} $X$ over $A$ is a right Hilbert module over $A$ with a left action given by a $*$-homomorphism $\phi_X \colon A \to \L X$.
We write $\L X$ and $\K X$ for the adjointable operators and the compact operators of $X$, respectively.
For two C*-corresponden\-ces $X, Y$ over the same $A$ we write $X \otimes_A Y$ for the balanced tensor product over $A$.
We say that $X$ is unitarily equivalent to $Y$ (symb. $X \simeq Y$) if there is a surjective adjointable operator $U \in \L(X,Y)$ such that $\sca{U \xi, U \eta} = \sca{\xi, \eta}$ and $U (a \xi b) = a U(\xi) b$ for all $\xi, \eta \in X$ and $a,b \in A$.

A \emph{representation} $(\pi,t)$ of a C*-correspondence is a left module map that preserves the inner product.
Then $(\pi,t)$ is automatically a bimodule map.
Moreover there exists a $*$-homomorphism $\psi$ on $\K X$ such that $\psi(\theta_{\xi, \eta}) = t(\xi) t(\eta)^*$ for all $\theta_{\xi, \eta} \in \K X$. When $\pi$ is injective, then both $t$ and $\psi$ are isometric.

The \emph{Toeplitz-Pimsner algebra $\T_X$} is the universal C*-algebra with respect to the representations of $X$.
The \emph{Cuntz-Pimsner algebra $\O_X$} is the universal C*-algebra with respect to the representations that satisfy in addition $\pi(a) = \psi(\phi_X(a))$ for all $a \in J_X$, for Katsura's ideal
\[
J_X := \ker\phi_X^\perp \bigcap \phi_X^{-1}(\K X).
\]

\subsection{Product systems}

All groups we consider are discrete.
Let $P$ be a sub-semigroup of a group $G$ such that $P \cap P^{-1} = \{e\}$.
Then $P$ defines a partial order on $G$ given by: $g \leq h$ if and only if $g^{-1} h \in P$.
The pair $(G,P)$ is a \emph{quasi-lattice ordered group} if any two elements $p, q \in G$ with common upper bound in $P$ have a least common upper bound $p\vee q$ in $P$.
We write $p\vee q = \infty$ when $p,q$ have no common upper bound in $P$, and write $p \vee q < \infty$ otherwise.
A set $F \subseteq P$ is called {$\vee$-closed} if $p \vee q \in F$ whenever $p, q \in F$ with $p \vee q <\infty$.

From now on fix $(G,P)$ be a quasi-lattice ordered group.
A \emph{product system $X$ over $P$} is a family $\{X_p \mid p \in P\}$ of C*-correspondences over the same C*-algebra $A$ such that
\begin{enumerate}
\item $X_e = A$;
\item there are multiplication rules $X_p \otimes_A X_q \simeq X_{pq}$ for every $p, q \in P \setminus \{e\}$.
\end{enumerate}
We will suppress the use of symbols for the multiplication rules.
Hence we write $\xi_p \xi_q$ for the image of $\xi_p \otimes \xi_q$ and so
\[
\phi_{pq}(a)(\xi_p \xi_q) = (\phi_p(a) \xi_p) \xi_q \foral a \in A \text{ and } \xi_p \in X_p, \xi_q \in X_q.
\]
More generally, the product system structure gives maps
\[
i_{p}^{pq} \colon \L X_{p} \to \L X_{pq}
\; \textup{ such that } \;
i_{p}^{pq}(S) (\xi_{p} \xi_{q})
=
(S \xi_{p}) \xi_{q}.
\]
Following Fowler's work \cite{Fow02}, a product system is called \emph{compactly aligned} if
\[
i_{p}^{p \vee q}(S) i_{q}^{p \vee q}(T) \in \K X_{p \vee q} \text{ whenever } S \in \K X_{p}, T \in \K X_{q} \text{ and } p \vee q < \infty.
\]
A \emph{Nica-covariant representation $(\pi, t)$} of a product system $X$ consists of a family of representations $(\pi, t_{p})$ of $X_{p}$ over $A$ that respect the multiplication of the product system in the sense that:
\[
t_{p q}(\xi_p \xi_q) = t_p(\xi_p) t_q(\xi_q) \foral \xi_p \in X_p, \xi_q \in X_q,
\]
and satisfy the \emph{Nica-covariance} for $S \in \K X_{p}$ and $T \in \K X_{q}$:
\[
\psi_{p}(S) \psi_{q}(T) = 
\begin{cases}
\psi_{p \vee q} (i_{p}^{p \vee q}(S) i_{q}^{p \vee q}(T)) & \text{ if } p\vee q < \infty, \\
0 & \text{ if } p \vee q = \infty.
\end{cases}
\]
In particular if $p \vee q = \infty$ then $t_p(\xi_p)^* t_q(\xi_q) = 0$, whereas if $p \vee q < \infty$ then
\[
t_p(\xi_p)^* t_q(\xi_q) \in \ol{t_r(X_r) t_s(X_s)^*} \qfor r = p^{-1}(p \vee q), s = q^{-1}( p \vee q).
\]
Hence the C*-algebra $\ca(\pi,t)$ generated by $\pi(A)$ and $t_p(X_p)$ can be written as
\[
\ca(\pi,t) = \ol{\spn}\{t_p(\xi_p) t_q(\xi_q)^* \mid \xi_p \in X_p, \xi_q \in X_q \textup{ and } p,q \in P\}.
\]
For a finite $F \subseteq P$ that is $\vee$-closed we write
\[
B_F := \spn\{ \psi_{p}(k_{p}) \mid k_p \in \K X_p, p \in F\}.
\]
It follows from Fowler's work \cite{Fow02} that the $B_F$ are closed C*-subalgebras of $\ca(\pi,t)$.
Moreover we write
\[
B_{(e, \infty]} := \ol{\spn} \{ \psi_p(k_p) \mid k_p \in \K X_p, e \neq p \in P \}
\qand
B_{[e, \infty]} := \pi(A) + B_{(e, \infty]}.
\]
We refer to these sets as the \emph{cores} of the representation $(\pi,t)$.

The \emph{Toeplitz-Nica-Pimsner} algebra $\N\T(X)$ is the universal C*-algebra generated by $A$ and $X$ with respect to the representations of $X$.
The Fock space representation of Fowler \cite{Fow02} ensures that $A$, and thus $X$, embeds isometrically in $\N\T(X)$.
In short, let $\F(X) = \sum^\oplus_{q \in P} X_q$ and for $a \in A$ and $\xi_p \in X_p$ define
\[
\pi(a) \xi_q = \phi_q(a) \xi_q
\qand
t(\xi_p) \xi_q = \xi_p \xi_q
\qforal
\xi_q \in X_q.
\]
Then the Fock representation $(\pi, t)$ is Nica-covariant and by taking the compression at the $(e, e)$-entry we see that $\pi$, and thus $t$, is injective.

Among the quotients of $\N\T(X)$ there is one of particular importance that generalizes the one-dimensional Cuntz-Pimsner algebra and was identified by Sims-Yeend \cite{SY11}.
Let 
\[
I_{e} := A 
\qand 
I_{q} = \bigcap \{ \ker \phi_{s} \mid s \leq q \}
\text{ for } e \neq q \in P,
\]
all of which are ideals in $A$. 
For $r \in P$, let
\[
\widetilde{X}_{r} = \bigoplus \{ X_{p} I_{q} \mid p, q \in P \text{ such that } p q = r \},
\]
and write $\wt\phi_{r}$ for the left action on $\widetilde{X}_{r}$. 
The product system $X$ is called \emph{$\wt\phi$-injective} if every $\wt\phi_r$ is injective.
Consequently, for $p \leq r$ we obtain a $*$-homomor\-phism 
\[
\widetilde{i}_{p}^{r} \colon \L X_{p} \rightarrow \L \widetilde{X}_{r}
\; \textup{ with } \;
\widetilde{i}_{p}^{r} := \oplus \{ i_{p}^{q} \mid p \leq q \leq r \}.
\]
If $p \not\leq r$ then we set $\widetilde{i}_{p}^{r} = 0$.
For a finite set $F \subset P$ we say that a collection $\{k_p\}_{p \in F}$ satisfies
\begin{equation}\label{eq:ast}
\sum \{ \, \widetilde{i}_{p}^{r}(k_{p}) \mid p \in F \, \} = 0 \textup{ for large $r$},
\end{equation}
if for every $s \in P$ there exists $s \leq x \in P$ such that equation (\ref{eq:ast}) holds for all $r \geq x$.
Let $\F_{\textup{CNP}}$ be the set of all such $\{k_p\}_{p \in F}$.
A Nica-covariant representation $(\pi,t)$ of $X$ will be called \emph{Cuntz-Nica-Pimsner-covariant} if:
\[
\sum \{ \psi_{p}(k_{p}) \mid p \in F \} = 0 \foral \{k_p\}_{p \in F} \in \F_{\textup{CNP}}.
\]
The \emph{Cuntz-Nica-Pimsner} algebra $\N\O(X)$ of Sims-Yeend is the universal C*-algebra generated by $A$ and $X$ with respect to the CNP-representations of $X$.
By using the augmented Fock space, Sims-Yeend \cite[Theorem 4.1]{SY11} show that if $X$ is $\wt\phi$-injective then $A$, and thus $X$, embeds isometrically in $\N\O(X)$.
In particular, this is the case for pairs $(G,P)$ with the property that every bounded subset of $P$ contains a maximal element \cite[Lemma 3.15]{SY11}.
For $\F \subset \F_{\textup{CNP}}$ we can form the quotient of $\N\T(X)$ be the universal C*-algebra with respect to the representations that are CNP on $\F$, i.e.
\[
\sum \{ \psi_{p}(k_{p}) \mid p \in F \} = 0 \foral \{k_p\}_{p \in F} \in \F.
\]
Such a quotient will be called \emph{$\F$-CNP-relative} or just \emph{CNP-relative}.
It follows that CNP-relative quotients factor through $\N\T(X) \to \N\O(X)$.

\section{Short exact sequences}

We will require some technical lemmas about short exact sequences.
Recall that for C*-algebras $A$ and $D$, we denote by $A\otimes D$ (resp. $A \otimes_{\max} D$) the \emph{minimal} (resp. \emph{maximal}) tensor product of $A$ and $D$, and by $A \ominus D$ the kernel of the natural surjection $\pi_{A,D} \colon A\otimes_{\max} D \rightarrow A \otimes D$. 
By definition 
\[
\pi_{B,D} \circ (\varphi \otimes_{\max} \id_D) = (\varphi \otimes \id_D) \circ \pi_{A,D}
\]
for any $*$-homomorphism $\varphi \colon A \to B$.
Thus the restriction of $\varphi \otimes_{\max} \id_D$ to $A \ominus D$ induces a map $\varphi \ominus \id_D \colon A\ominus D \to B \ominus D$.
Recall that $\varphi$ is \emph{nuclear} if and only if $\varphi \ominus \id_D = 0$ for any C*-algebra $D$.
Moreover $A$ is \emph{nuclear}, if and only if $\id_A$ is nuclear, if and only if $A \ominus D = (0)$ for any C*-algebra $D$.
We will constantly refer to some central results from \cite[Appendices]{Kat04}.
We will also need some variation of \cite[Proposition A.6]{Kat04}. 
For one direction we have the following.

\begin{proposition} \label{P:nuc-quot}
Let $A,A'$ be C*-algebras and let the ideals $I \lhd A$ and $I'\lhd A'$. 
Suppose we have the following commutative diagram of short exact sequences
\[
\xymatrix{
0 \ar[r] & I \ar[r] \ar[d]^{\varphi_0} & A \ar[r] \ar[d]^{\varphi} & A / I \ar[r] \ar[d]^{\widetilde{\varphi}} & 0\\
0 \ar[r] & I' \ar[r] & A' \ar[r] & A' / I' \ar[r] & 0
}
\]
where $\varphi \colon A \rightarrow A'$ is an injective $*$-homomorphism that satisfies $\varphi(I) \subseteq I'$, $\widetilde{\varphi} \colon A / I \rightarrow A' / I'$  is the induced map and $\varphi_0 :=\varphi|_{I}$. 
If $\varphi \colon A\rightarrow A'$ is nuclear, then $\varphi_0$ and $\widetilde{\varphi}$ are both nuclear.
\end{proposition}

\begin{proof}
Fix a C*-algebra $D$. Since $\varphi$ is an injective nuclear map then $A$ is exact.
Hence $\varphi \ominus \id_D = 0$, and by \cite[Lemma A.5]{Kat04} we get
\begin{align*}
\xymatrix@C=1.3cm{
0 \ar[r] & I \ominus D \ar[r] \ar[d]^{\varphi_0 \ominus \id_D} & A \ominus D \ar[r]^{\pi \ominus \id_D} \ar[d]^{0} & (A / I) \ominus D \ar[r] \ar[d]^{\widetilde{\varphi} \ominus \id_D} & 0\\
0 \ar[r] & I' \ominus D \ar[r] & A' \ominus D \ar[r] & (A' / I') \ominus D .
}
\end{align*}
By commutation of the left square and exactness on $I' \ominus D$, we must have that $\varphi_0 \ominus \id_D =0$ so that $\varphi_0$ is nuclear. 
By commutation of the right square and exactness on $(A/I)\ominus D$, we must have that $\wt\varphi \ominus \id_D = 0$ so that $\wt\varphi$ is nuclear. 
\end{proof}

The ideals $I, I'$ we will be using are in good position in the sense that $I \subseteq I'$ and there is an approximate identity for $I'$, such that multiplying it by an element in $A$ lands us inside $I$.
This will give \emph{a} converse to Proposition \ref{P:nuc-quot}.
We require the next technical lemma.

\begin{lemma} \label{L:max-ideal-inc}
Let $A, A'$ be C*-algebras and let the ideals $I \lhd A$ and $I'\lhd A'$. 
Suppose that $\varphi \colon A \rightarrow A'$ is an injective $*$-homomorphism such that $\varphi(I) \subseteq I'$ and there exists an approximate identity $(e_{\la})_\la$ of $I'$ such that $\varphi(a)e_{\la} \in \varphi(I)$ for all $a\in A$. 
Then 
\begin{equation*}
\varphi \otimes_{\max} \id_D (A \otimes_{\max} D) \bigcap I' \otimes_{\max} D 
\subseteq 
\varphi \otimes_{\max} \id_D (I \otimes_{\max} D)
\end{equation*}
for any C*-algebra $D$.
\end{lemma}

\begin{proof}
Let $y \in \varphi \otimes_{\max} \id_D (A \otimes_{\max} D) \bigcap I' \otimes_{\max} D$, and choose $x \in A \otimes_{\max} D$ such that $\varphi \otimes_{\max} \id_D(x) = y$. 
By approximation from $A\odot D$ we can write
\begin{align*}
x = \lim_{\kappa} \sum_{i=1}^{n_{\kappa}} a_i^{(\kappa)} \otimes d_i^{(\kappa)} \qfor a_i^{(\kappa)} \in A, d_i^{(\kappa)} \in D.
\end{align*}
Let $(f_\mu)_\mu$ be an approximate identity of $D$.
As $y \in I' \otimes_{\max} D \subseteq A' \otimes_{\max} D$ we then get
\begin{align*}
y
& = 
\lim_{\la, \mu} y (e_{\la}\otimes f_\mu) 
 = 
\lim_{\la, \mu} \varphi \otimes_{\max} \id_D(x) \cdot (e_{\la} \otimes f_\mu) 
 = 
\lim_{\la, \mu} \lim_{\kappa} \sum_{i=1}^{n_{\kappa}} \varphi(a_i^{(\kappa)}) e_{\la} \otimes d_i^{(\kappa)} f_\mu.
\end{align*}
By assumption $\varphi(a_i^{(\kappa)})e_{\la} \in \varphi(I)$, and thus $y \in \varphi \otimes_{\max} \id_D (I \otimes_{\max} D)$.
\end{proof}

\begin{proposition} \label{P:nuc-ext}
Let $A, A'$ be C*-algebras and let the ideals $I \lhd A$ and $I'\lhd A'$. 
Suppose we have the following commutative diagram of short exact sequences
\[
\xymatrix{
0 \ar[r] & I \ar[r] \ar[d]^{\varphi_0} & A \ar[r] \ar[d]^{\varphi} & A / I \ar[r] \ar[d]^{\widetilde{\varphi}} & 0\\
0 \ar[r] & I' \ar[r] & A' \ar[r] & A' / I' \ar[r] & 0
}
\]
where $\varphi \colon A \rightarrow A'$ is an injective $*$-homomorphism that satisfies $\varphi(I) \subseteq I'$, $\widetilde{\varphi} \colon A / I \rightarrow A' / I'$  is the induced map and $\varphi_0 :=\varphi|_{I}$. 
Suppose further that there exists an approximate identity $(e_{\la})_\la$ of $I'$ such that $\varphi(a)e_{\la} \in \varphi(I)$ for all $a\in A$. 
If $\varphi_0$ and the induced map $\widetilde{\varphi}$ are nuclear, then so is $\varphi$.
\end{proposition}

\begin{proof}
Fix a C*-algebra $D$. 
Since $\varphi_0$ and $\widetilde{\varphi}$ are nuclear, we see that $\varphi_0 \ominus \id_D = 0$ and $\widetilde{\varphi}\ominus \id_D = 0$. 
Then exactness of the maximal tensor product and \cite[Lemma A.5]{Kat04} yields
\begin{align*}
\xymatrix@C=1.5cm{
0 \ar[r] & I \ominus D \ar[r] \ar[d]^{0} & A \ominus D \ar[r] \ar[d]^{\varphi \ominus \id_D} & (A / I) \ominus D \ar[d]^{0} \\
0 \ar[r] & I' \ominus D \ar[r] & A' \ominus D \ar[r] & (A' / I') \ominus D.
}
\end{align*}
Let $x \in A \ominus D$.
By commutation of the right square and exactness on $A' \ominus D$, we see that $\varphi \otimes_{\max} \id_D(x) \in I' \ominus D$. 
By Lemma \ref{L:max-ideal-inc} there exists a $c \in I \otimes_{\max} D$ such that $\varphi_0 \otimes_{\max} \id_D(c) = \varphi \otimes_{\max} \id_D(x)$. 
However, since $\varphi \otimes_{\max} \id_D(x) \in I' \ominus D$, we get
\begin{align*}
(\varphi_0 \otimes \id_D) \circ \pi_{I,D} (c)
=
\pi_{I',D} \circ (\varphi_0 \otimes_{\max} \id_D) (c) 
=
\pi_{I',D} \circ (\varphi \otimes_{\max} \id_D)(x)
=
0.
\end{align*}
Injectivity of $\varphi_0 \otimes \id_D$ thus implies that $c \in \ker \pi_{I,D} \equiv I \ominus D$. 
Hence we get that
\[
\varphi \otimes_{\max} \id_D(x) = \varphi_0 \otimes_{\max} \id_D (c) = \varphi_0 \ominus \id_D (c) = 0.
\]
As this holds for all $x \in A \ominus D$ we conclude that $\varphi \ominus \id_D = 0$, and thus $\varphi$ is nuclear.
\end{proof}

\section{$\bZ_+^N$-product systems}\label{S:ZN}

For this section fix a product system $X$ over $(\bZ^N, \bZ_+^N)$ with coefficients in $A$.
A representation $(\pi,t)$ of $X$ admits a gauge action if there is a point-norm continuous homomorphism $\beta \colon \mathbb{T}^N \rightarrow \Aut(\ca(\pi,t))$ such that
\[
\be_{\un{z}}(t_\un{n}(\xi_{\un{n}})) = \un{z}^{\un{n}} \, t_{\un{n}}(\xi_{\un{n}})
\foral \xi_{\un{n}} \in X_{\un{n}} 
\qand 
\be_{\un{z}}(\pi(a)) = \pi(a)
\foral a \in A.
\]
Then the fixed point algebra $\ca(\pi,t)^\be$ coincides with $B_{[\un{0}, \infty]}$.
By universality, the C*-algebra $\N\T(X)$ attains a gauge action.
Moreover it follows that the Fock representation defines a faithful $*$-representation of $\N\T(X)$.

We will investigate nuclearity and exactness for quotients of $\N\T(X)$ that inherit this gauge action and attain a faithful copy of $A$, such as the Cuntz-Nica-Pimsner algebra $\N\O(X)$.
Before we proceed we have to introduce some notation and a definition.
The free generators of $\bZ_+^N$ will be denoted by $\bo{1}, \bo{2}, \dots, \bo{N}$.
We denote \emph{the support of $\un{n}$} by
\[
\supp \un{n} := \{i \in \{1, \dots, N\} \mid n_i \neq 0\}.
\]
For any $F \subseteq \{1, \dots, N\}$ we write $\un{1}_F := \sum_{i \in F} \Bi$.

\begin{definition}
Let $(\pi,t)$ be a Nica-covariant representation of a product system $X$ over $(\bZ^N, \bZ_+^N)$ with coefficients in $A$.
A core $B_{[\un{m},\un{m}+\un{m}']}$ of $\ca(\pi,t)$ is called a \emph{$d$-dimensional box} if $|\supp{\un{m}'}| = d$. 
\end{definition}

Let $\un{m}' \in \bZ_+^N$ with $|\supp{\un{m}'}| = d + 1$. 
By using a fixed $i\in \supp{\un{m}'}$ we may write
\[
B_{[\un{m},\un{m}+\un{m}']} 
= 
\sum_{k=0}^{m_i'} B_{[\un{m} + k \bo{i}, \un{m} + k \bo{i} + (\un{m}' - m_i' \bo{i})]}
\]
as a sum of $m_i'+1$ boxes of dimension $d$. 
Moreover, for each $0 \leq n \leq m_i' - 1$, the box
\[
B_{[\un{m} + (n+1) \bo{i}, \un{m} + \un{m}']} 
= 
\sum_{k=n +1}^{m_i'} B_{[\un{m} + k \bo{i}, \un{m} +k \bo{i} + (\un{m}' - m_i' \bo{i})]}
\]
is an ideal inside $B_{[\un{m} + n\bo{i}, \un{m} + \un{m}']}$.
Let $e_{\la} := \psi_{\un{m}}(k_{\un{m}, \la})$ for an approximate identity $(k_{\un{m}, \la})$ in $\K X_{\un{m}}$. 
Then $(e_\la)_\la$ is an approximate identity for $B_{[\un{m},\un{m}]}, B_{[\un{m},\un{m} + \un{m}']}$ and $B_{[\un{m},\infty]}$ when $\un{m} > \un{0}$.

Let us first deal with exactness and thus recapture \cite[Corollary 4.18]{Fle17}.
Our method is different from the semi-direct decomposition of Deaconu \cite{Dea07} and Fletcher \cite{Fle17}.

\begin{theorem} \label{T:exact}
Let $X$ be a compactly aligned product system over $(\bZ^N, \bZ_+^N)$ with coefficients in $A$. 
Let $\N\P(X)$ be a quotient of $\N\T(X)$ such that $A \hookrightarrow \N\P(X)$ isometrically.
Then $\N\P(X)$ is exact if and only if $A$ is exact.
\end{theorem}

\begin{proof}
If $\N\P(X)$ is exact then $A$ is exact since exactness passes to subalgebras.
As $\N\P(X)$ is a quotient of $\N\T(X)$, and as exactness passes to quotients, it remains to show that exactness of $A$ implies exactness of $\N\T(X)$; equivalently that $\N\T(X)^\be$ is exact by \cite[Proposition A.13]{Kat04}.

Let $(\pi,t)$ be the Fock representation with cores $B_{[\un{m}, \un{m}+\un{m}']}$.
We first show that each $B_{[\un{m}, \un{m}+\un{m}']}$ is the linear direct sum of the C*-algebras $B_{[\un{n}, \un{n}]}$ for $\un{m} \leq \un{n} \leq \un{m}+\un{m}'$. 
Indeed suppose there are $k_\un{n} \in \K X_{\un{n}}$ such that
\[
\sum \{ \psi_{\un{n}}(k_{\un{n}}) \mid \un{m} \leq \un{n} \leq \un{m}+\un{m}' \} = 0,
\]
and we will show that every $k_{\un{n}}$ is zero.
To reach a contradiction let $\un{\ell} \in [\un{m}, \un{m}+\un{m}']$ be minimal such that $k_{\un{\ell}} \neq 0$.
Let $p_{\un{\ell}} \colon \F X \to X_{\un{\ell}}$ be the projection onto $X_{\un{\ell}}$.
Minimality of $\un{\ell}$ yields $p_{\un{\ell}} \psi_{\un{n}}(k_{\un{n}}) p_{\un{\ell}} = 0$ for all $\un{n} \neq \un{\ell}$, and thus we get the contradiction
\[
k_{\un{\ell}} 
= 
p_{\un{\ell}} \psi_{\un{\ell}}(k_{\un{\ell}}) p_{\un{\ell}}
=
p_{\un{\ell}} \left( \sum \{ \psi_{\un{n}}(k_{\un{n}}) \mid \un{m} \leq \un{n} \leq \un{m}+\un{m}' \} \right) p_{\un{\ell}}
=
0.
\]

Now, the fixed point algebra $\N\T(X)^\be$ is the inductive limit of the cores $B_{[\un{m}, \un{m}+\un{m}']}$.
Thus in order to finish the proof it suffices to show that each such core is exact.
By \cite[Proposition B.7]{Kat04}, each $\K X_{\un{n}}$ is exact and thus so is every $B_{[\un{n}, \un{n}]} \simeq \psi_{\un{n}}(\K X_{\un{n}})$.
For the inductive step, suppose that all $d$-dimensional boxes are exact.
Let ${B}_{[\un{m},\un{m}+\un{m}']}$ be now a $(d+1)$-dimensional box.
Let $i \in \supp \un{m}'$; then by the inductive hypotheses every $d$-dimensional box
\[
B_n:= B_{[\un{m} + n \bo{i}, \un{m} + n \bo{i} + (\un{m}' - m_i' \bo{i})]} 
\qfor 0 \leq n \leq m_i',
\]
is exact. 
Also let the algebras $A_{m_i'} := B_{m_i'}$ and
\[
A_n := \sum_{k=n}^{m_i'} B_k,
\qfor
0 \leq n \leq m_i' - 1,
\]
so that $A_{n+1}$ is an ideal inside $A_n$. 
Since $A_n = B_n + A_{n+1}$ is a linear direct sum, we see that $A_n / A_{n+1} \simeq B_n$, and we obtain the following \emph{split} exact sequence
\begin{align*}
\xymatrix{
0 \ar[r] & A_{n+1} \ar[r] & A_n \ar[r] & B_n \ar[r] & 0.
}
\end{align*}
As exactness passes to extensions in split exact sequences, we may apply for $n = m_i' - 1$ and obtain that $A_{m_i' -1}$ is exact.
Now we may induct on $n$ and get that every $A_n$ is exact; in particular so is $A_0 = B_{[\un{m}, \un{m}+\un{m}']}$.
\end{proof}

\begin{remark}
Similar arguments to Theorem \ref{T:exact} suffice to show that $A$ is nuclear if and only if $\N\T(X)$ is nuclear.
Consequently, if $A$ is nuclear then every quotient $\N\P(X)$ of $\N\T(X)$ is nuclear.
However the converse does not hold even for $(\mathbb{Z},\mathbb{Z}_+)$.
For a counterexample where $\O_X$ is nuclear but $A$ is not see \cite[Example 7.7]{Kat04}.
\end{remark}

Our next objective is to show that nuclearity of $\N\O(X)$ is equivalent to the embedding $A \hookrightarrow \N\O(X)$ being nuclear.
We can show this for any $\N\P(X)$ that is a $\bT^N$-equivariant quotient of $\N\T(X)$ and the emdedding $A \hookrightarrow \N\P(X)$ is isometric.
Notice that $\N\P(X)$ inherits the gauge action of $\N\T(X)$.

\begin{theorem}\label{T:nuc NO}
Let $X$ be a compactly aligned product system over $(\bZ^N, \bZ_+^N)$ with coefficients in $A$. 
Let $\N\P(X)$ be a $\bT^N$-equivariant quotient of $\N\T(X)$ such that $A \hookrightarrow \N\P(X)$ isometrically. 
Set $I := A \cap B_{(\un{0}, \infty]}$.
Then the following are equivalent:
\begin{enumerate}
\item
$A / I$ is nuclear and the embedding $I \hookrightarrow {B}_{(\un{0},\infty]}$ is nuclear;
\item
The embedding $A \hookrightarrow \N\P(X)^{\beta}$ is nuclear;
\item
The embedding $A \hookrightarrow \N\P(X)$ is nuclear;
\item
$\N\P(X)^{\beta}$ is nuclear;
\item
$\N\P(X)$ is nuclear.
\end{enumerate}
\end{theorem}

\begin{proof}
For the equivalence of items (i) and (ii), consider the following commutative diagram of short exact sequences
\begin{align*}
\xymatrix{
0 \ar[r] & I \ar[r] \ar[d] & A \ar[r] \ar[d] & A / I \ar[r] \ar[d] & 0\\
0 \ar[r] & {B}_{(\un{0},\infty]} \ar[r] & {B}_{[\un{0},\infty]} \ar[r] & {B}_{[\un{0},\infty]} / {B}_{(\un{0},\infty]} \ar[r] & 0.
}
\end{align*}
However by definition we have that
\[
{B}_{[0,\infty]} / {B}_{(\un{0},\infty]} \simeq A / (A \cap {B}_{(\un{0},\infty]}) = A/ I.
\]
Therefore \cite[Proposition A.6]{Kat04} implies that the embedding $A \hookrightarrow {B}_{[\un{0},\infty]}$ is nuclear if and only if $A/ I$ and the embedding $I \hookrightarrow {B}_{(\un{0},\infty]}$ are both nuclear.

Items (iv) and (v) are equivalent by \cite[Proposition A.13]{Kat04}.
Trivially item (v) implies item (iii).
Moreover since $A \subseteq \N\P(X)^\be = B_{[\un{0}, \infty]}$ we can compose with the conditional expectation to deduce that item (iii) implies item (ii).
Hence it remains to show that (ii) implies (iv).

Note that nuclearity of $A \hookrightarrow {B}_{[\un{0},\infty]}$ implies exactness for $A$, which in turn implies exactness of $\N\P(X)$ by Theorem \ref{T:exact}. 
Hence, all C*-algebras involved will be exact.
The plan is to show that nuclearity of $A\hookrightarrow {B}_{[\un{0},\infty]}$ implies nuclearity of the embedding ${B}_{[\un{m},\un{m}+\un{m}']} \hookrightarrow {B}_{[\un{m},\infty]}$ for any box ${B}_{[\un{m},\un{m}+\un{m}']}$, by inducting on the dimension $d=|\supp \un{m}'|$. 
With that in hand, we get that ${B}_{[\un{0},m \cdot\un{1}_F]}$ embeds nuclearly in ${B}_{[\un{0},\infty]}$ for every $m \in \bZ_+$ and $F \subseteq \{1, \dots, N\}$.
Therefore by taking the direct limit we can conclude that ${B}_{[\un{0},\infty]}$ is nuclear.

To this end suppose that $\N\P(X) = \ca(\pi,t)$.
For the base step of the induction, notice that the inclusion
\[
{B}_{[\un{m},\un{m}]} = \psi_{\un{m}}(X_{\un{m}}) \hookrightarrow {B}_{[\un{m},\infty]}
\]
is nuclear by applying \cite[Proposition B.8]{Kat04} for $X= X_{\un{m}}$ and $Y= \ol{t(X_{\un{m}}) \cdot B_{[\un{0},\infty]}}$. 
For the inductive step, suppose that all $d$-dimensional boxes $B_{[\un{m}, \un{m}+\un{m}']}$ embed into their corresponding ${B}_{[\un{m},\infty]}$ nuclearly.
Let ${B}_{[\un{m},\un{m}+\un{m}']}$ be now a $(d+1)$-dimensional box.
Fix an $i \in \supp \un{m}'$ and write
\[
{B}_{[\un{m},\un{m}+\un{m}']}
=
\sum_{k=0}^{m_i'} B_k
\qfor
B_k := {B}_{[\un{m} + k \Bi,\un{m} + k\Bi + (\un{m}'- m_i' \Bi)]},
\]
where every $B_k$ is a $d$-dimensional box.
We will reconstruct ${B}_{[\un{m},\un{m}+\un{m}']}$ inductively, and so we define
\[
A_n := \sum_{k=n}^{m_i'} B_k
\qand A_n^\infty:= {B}_{[\un{m} + n\Bi, \infty]}
\qfor 
0 \leq n \leq m_i'.
\]
It is clear that $A_{m_i'} = B_{m_i'}$ and $A_n = B_n + A_{n+1}$ for $0 \leq n \leq m_i' -1$.
By the inductive hypothesis on $d$-dimensional boxes the inclusions $B_n \hookrightarrow A_n^\infty$ are nuclear for all $0 \leq n \leq m_i'$.
We wish to induct on $n$ and show that the inclusions $A_n \hookrightarrow A_n^\infty$ are nuclear for all $0 \leq n\leq m_i'$, thus concluding nuclearity for $n =0$, i.e., for
\[
B_{[\un{m}, \un{m} + \un{m}']} = A_0 \hookrightarrow A_0^\infty = B_{[\un{m}, \infty]}.
\]

For the base step $n = m_i'$, by the inductive hypothesis on $d$-dimensional boxes, we have that the embedding $B_{m_i' - 1} \hookrightarrow A_{m_i' - 1}^\infty$ is nuclear.
Therefore by applying Proposition \ref{P:nuc-quot} to the following commutative diagram
\begin{equation*}
\xymatrix{
0 \ar[r] & B_{m_i'-1} \cap B_{m_i'} \ar[r] \ar[d] & B_{m'_i - 1} \ar[r] \ar[d] & B_{m_i' - 1}/ B_{m_i'-1} \cap B_{m_i'} \ar[r] \ar[d] & 0\\
0 \ar[r] & A_{m_i'}^{\infty} \ar[r] & A_{m_i' - 1}^{\infty} \ar[r] & A_{m_i'-1}^{\infty} / A_{m_i'}^{\infty} \ar[r] & 0
}
\end{equation*}
we derive that the map
\[
 A_{m_i' - 1}/ A_{m_i'} \simeq B_{m_i' - 1}/ B_{m_i'-1} \cap B_{m_i'} \to 
A_{m_i' - 1}^\infty/ A_{m_i'}^\infty
\]
is nuclear.
Furthermore, again by the inductive hypothesis for $d$-dimensional boxes, the inclusion
\[
A_{m_i'} = B_{m_i'} \hookrightarrow B_{[\un{m} + m_i' \Bi, \infty]} = A_{m_i'}^\infty
\]
is nuclear.
Now the approximate identity of ${B}_{[\un{m} + m_i' \Bi, \un{m} + m_i' \Bi]}$ is an approximate identity for both $A_{m_i'}$ and $A_{m_i'}^\infty$, and Nica-covariance yields
\[
{B}_{[\un{m} + m_i' \Bi, \un{m} + m_i' \Bi]} \cdot A_{m_i'-1} \subseteq A_{m_i'}.
\]
We may thus apply Proposition \ref{P:nuc-ext} to the diagram
\begin{equation*}
\xymatrix{
0 \ar[r] & A_{m_i'} \ar[r] \ar[d] & A_{m_i' - 1} \ar[r] \ar[d] & A_{m_i' - 1} / A_{m_i'} \ar[r] \ar[d] & 0\\
0 \ar[r] & A_{m_i'}^{\infty} \ar[r] & A_{m_i' - 1}^{\infty} \ar[r] & A_{m_i' - 1}^{\infty} / A_{m_i'}^{\infty} \ar[r] & 0
}
\end{equation*}
and get that the inclusion $A_{m_i'-1} \hookrightarrow A_{m_i'-1}^\infty$ is nuclear.
This proves the base step $n = m_i'$.

Now suppose that the embedding $A_{n+1} \hookrightarrow A_{n+1}^\infty$ is nuclear, for $n+1 \leq m_i'$.
We will show that the embedding $A_{n} \hookrightarrow A_{n}^\infty$ is also nuclear.
By the standing hypothesis we have that the middle arrow in the diagram
\begin{align*}
\xymatrix{
0 \ar[r] & A_{n+1} \cap B_n \ar[r] \ar[d] & B_n \ar[r] \ar[d] & B_n / A_{n+1} \cap B_n \ar[r] \ar[d] & 0\\
0 \ar[r] & A_{n+1}^{\infty} \ar[r] & A_n^{\infty} \ar[r] & A_n^{\infty} / A_{n+1}^{\infty} \ar[r] & 0.
}
\end{align*}
is nuclear.
As before, by Proposition \ref{P:nuc-quot} we get that the map
\[
A_{n}/A_{n+1} \simeq B_n / (A_{n+1} \cap B_n) \to A_n^\infty / A_{n+1}^\infty
\]
is nuclear.
Moreover the approximate identity of ${B}_{[\un{m} + (n+1) \Bi, \un{m} + (n+1) \Bi]}$ is an approximate identity for both $A_{n + 1}$ and $A_{n + 1}^\infty$, and Nica-covariance yields
\[
{B}_{[\un{m} + (n+1) \Bi, \un{m} + (n+1) \Bi]} \cdot A_{n} \subseteq A_{n+1}.
\]
The map $A_{n+1} \hookrightarrow A_{n+1}^\infty$ is nuclear by the inductive step, and thus we can apply Proposition \ref{P:nuc-ext} to the diagram
\begin{equation*}
\xymatrix{
0 \ar[r] & A_{n+1} \ar[r] \ar[d] & A_n \ar[r] \ar[d] & A_n / A_{n+1} \ar[r] \ar[d] & 0\\
0 \ar[r] & A_{n+1}^{\infty} \ar[r] & A_n^{\infty} \ar[r] & A_n^{\infty} / A_{n+1}^{\infty} \ar[r] & 0
}
\end{equation*}
and derive that its middle vertical embedding is nuclear.
This completes the proof.
\end{proof}

\begin{remark}
There is a second way to obtain Theorem \ref{T:nuc NO} for $\N\O(X)$ when $X$ is regular, i.e., when every $\phi_{\un{n}}$ is injective and $\phi_{\un{n}}(A) \subseteq \K X_{\un{n}}$.
For in this case there are injective mappings
\[
\phi_{\un{m}} \colon A \to \K X_{\un{m}}
\qand
\K X_{\un{m}} \to \K (X_{\un{m}} \otimes X_{\un{m}'}): k_{\un{m}} \mapsto k_{\un{m}} \otimes \id_{X_{\un{m}'}}.
\]
As in the one-variable case (that can be traced back to Pimsner \cite{Pim97}) it is not hard to show that $\N\O(X)^\be$ is the inductive limit of $\K X_{\un{n}}$ under these maps.
This follows by observing that $\N\O(X)$ is the universal C*-algebra generated by representations that are fiber-wise covariant in the sense of Katsura by \cite[Corollary 5.2]{SY11}, and thus
\[
B_{[\un{m}, \un{m} + \un{m}']} 
= 
\spn\{\psi_{\un{n}}(\K X_{\un{n}}) \colon \un{m} \leq \un{n} \leq \un{m} + \un{m}'\}
=
\psi_{\un{m} + \un{m}'}(\K X_{\un{m} + \un{m}'}).
\]
Therefore when $A \hookrightarrow \N\O(X)^\be$ is nuclear then the inclusions $\psi_{\un{n}}(\K X_{\un{n}}) \subseteq \N\O(X)$ are nuclear, and this yields nuclearity of $\N\O(X)^\be$. 
Fletcher \cite[Corollary 5.21]{Fle17} obtains the same result as an application of the elegant semi-direct decomposition of $X$.
\end{remark}

\begin{remark}
The ideal $A \cap B_{(\un{0}, \infty]}$ plays a fundamental role in Theorem \ref{T:nuc NO}, and there is a wide class of product systems for which we have a description.
In \cite{DK18b} we introduced the term \emph{strong compactly aligned} for product systems over $(\bZ^N, \bZ_+^N)$ if $\K X_{\un{n}} \otimes \id_{X_{\Bi}} \subseteq \K( X_{\un{n}} \otimes X_{\Bi})$ whenever $i \notin \supp \un{n}$. 
The method of solving polynomial equations of \cite[Section 4]{DK18b} yields that
\begin{equation}
A \cap B_{(\un{0}, \infty]} \subseteq (\bigcap_{i=1}^N \ker \phi_{\Bi})^\perp \cap (\bigcap_{i=1}^{N} \phi_{\Bi}^{-1}(\K X_{\Bi})),
\end{equation}
for any $\bT^N$-equivariant quotient $\N\P(X) = \ca(\pi,t)$ of $\N\T(X)$ that admits a faithful copy of $A$.
In particular in \cite{DK18b} it is shown that equality holds for $\N\O(X)$.
\end{remark}

\begin{remark}
Theorem \ref{T:nuc NO} is new even for the case of $N=1$, i.e., for the relative Cuntz-Pimsner algebras of Muhly-Solel \cite{MS98}.
Let us provide some background.

Let $X$ be a C*-correspondence over $A$ and fix $J \subseteq \phi^{-1}_X(\K X)$.
The \emph{$J$-relative Cuntz-Pimsner algebra $\O(J,X)$} is the universal C*-algebra with respect to the representations $(\pi,t)$ that satisfy
\[
\pi(a) = \psi(\phi_X(a)) \foral a \in J.
\]
By \cite[Lemma 1.2]{KP13} the embedding $A \hookrightarrow \O(J,X)$ is isometric if and only if $J \subseteq J_X$.
Moreover in \cite{Kak16} it is shown that every $\bT$-equivariant quotient of $\T_X$ that admits an isometric copy of $A$ is a relative Cuntz-Pimsner algebra in the following sense.
Let $\T_X = \ca(\pi,t)$ and $q \colon \T_X \to \T_X/I$.
Then $\T_X/I \simeq \O(J,X)$ for $J = \{a \in A \mid  q\pi(a) \in q\psi(\K X) \}$.
So in fact $J = A \cap B_{(0,\infty]}$.

Now we can see that an application of Theorem \ref{T:nuc NO} gives the following equivalent items for $\O(J, X)$ when $J \subseteq J_X$:
\begin{enumerate}
\item
$A / J$ is nuclear and the embedding $J \hookrightarrow {B}_{(\un{0},\infty]}$ is nuclear;
\item
The embedding $A \hookrightarrow \O(J, X)$ is nuclear;
\item
$\O(J, X)$ is nuclear.
\end{enumerate}
\end{remark}

\section{$\bZ_+^N$-controlled product systems}\label{S:ZN-con}

Recall from \cite{CL07} that a \emph{controlled map} $\vartheta \colon (G,P) \to (G',P')$ between quasi-lattice ordered groups is an order preserving homomorphism such that:
\begin{enumerate}
\item[(C1)] the restriction $\vartheta|_P$ is finite-to-one;
\item[(C2)] for all $p,q \in P$ with $p \vee q < \infty$ we have $\vartheta(p) \vee \vartheta(q) = \vartheta(p \vee q)$.
\end{enumerate}
It then follows that:
\begin{enumerate}
\item[(C3)] $\vartheta^{-1}(e') \cap P = \{e\}$;
\item[(C4)] if $p \vee q < \infty$ and $\vartheta(p) \leq \vartheta(q)$ then $p \leq q$;
\item[(C5)] if $p \vee q < \infty$ and $\vartheta(p) = \vartheta(q)$ then $p=q$.
\end{enumerate}

In this section we consider product systems over $(G,P)$ for which there exists a controlled map $\vartheta\colon (G, P) \to (\bZ^N, \bZ_+^N)$.
We say that a Nica-covariant representation $(\pi,t)$ of $X$ over $P$ admits a \emph{$\bT^N$-gauge action via $\vartheta$} if there is a point-norm continuous family homomorphism $\beta \colon \mathbb{T}^N \rightarrow \Aut(\ca(\pi,t))$ such that
\[
\be_{\un{z}}(t_p(\xi_p)) = \un{z}^{\vartheta(p)} \, t_{p}(\xi_{p})
\foral \xi_{p} \in X_{p}
\qand 
\be_{\un{z}}(\pi(a)) = \pi(a)
\foral a \in A.
\]
Universality implements a $\bT^N$-gauge action $\be$ on $\N\T(X)$.
A usual gauge-action argument gives that the Fock representation defines a faithful representation of $\N\T(X)$.
This will become clear once we identify the appropriate boxes here.
To this end let a pair $(\pi,t)$ that admits a $\bT^N$-gauge action via $\vartheta$.
If $p,q \in P$ with $\vartheta(p) = \vartheta(q)$ then by (C5) either $p = q$ or $p \vee q = \infty$; thus Nica-covariance yields
\[
t_{p}(\xi_{p})^* t_{q}(\eta_{q})
=
\de_{p,q} \pi(\sca{\xi_{p}, \eta_{q}}).
\]
For $\un{n} \in \bZ_+^N$ we thus define the induced sum of \emph{orthogonal} C*-correspondences
\[
X_{\un{n}} := \sum\{ t_p(X_p) \mid \vartheta(p) = \un{n} \},
\]
which becomes a C*-correspondence over $A$.
Notice that the sum is finite as $\vartheta|_P$ is finite-to-one.
The compact operators ideal $\K X_{\un{n}}$ is then $*$-isomorphic to the matrix algebra
\[
B_{[\un{n}, \un{n}]} := 
\sum \{ t_p(X_{p}) t_q(X_{q})^* \mid \vartheta(p) = \un{n} = \vartheta(q) \}.
\]
We define
\[
B_{[\un{m}, \un{m} + \un{m}']} := \sum \{ B_{[\un{n}, \un{n}]} \mid \un{n} \in [\un{m}, \un{m} + \un{m}'] \},
\]
and note that the fixed point algebra $\ca(\pi,t)^\be$ is the union of those boxes.

\begin{theorem} \label{T:ex con}
Let $\vartheta \colon (G,P) \to (\bZ^N, \bZ_+^N)$ be a controlled map and let $X$ be a compactly aligned product system over $(G,P)$ with coefficients in $A$. 
Let $\N\P(X)$ be a quotient of $\N\T(X)$ such that $A \hookrightarrow \N\P(X)$ isometrically. 
Then $\N\P(X)$ is exact if and only if $A$ is exact.
\end{theorem}

\begin{proof}
As with Theorem \ref{T:exact} we just need to check for $\N\T(X)$.
Consider first the Fock representation $(\pi,t)$.
We may apply \cite[Proposition B.7]{Kat04} for $X_{\un{n}} := \sum\{ t_p(X_p) \mid \vartheta(p) = \un{n} \}$ and deduce that the cores $B_{[\un{n}, \un{n}]}$ of $\N\T(X)$ are exact.
Now we may follow the proof of Theorem \ref{T:exact} verbatim and derive that the fixed point algebra
\[
\ca(\pi,t)^\be = \ol{\spn} \{ t(X_p) t(X_q)^* \mid \vartheta(p) = \vartheta(q) \}
\]
(and thus $\ca(\pi,t)$) is exact.
The only difference is that we have to consider separately the projections $p_q \colon \F X \to X_q$.
To this end let $\psi_{r, r'} \colon \K (X_r, X_{r'}) \to \B(H)$ be the induced map such that $\psi_{r, r'}(\theta_{\xi_{r'}, \xi_{r}}) = t(\xi_{r'}) t(\xi_{r})^*$, and suppose that
\[
f := \sum \{ \psi_{r_1, r_2}(k_{r_1, r_2}) \mid k_{r_1, r_2} \in \K(X_{r_1}, X_{r_2}), \vartheta(r_1) = \vartheta(r_2) \in [\un{m}, \un{m} + \un{m}'] \} = 0.
\]
Let $\un{\ell}$ be minimal such that $k_{q_1, q_2} \neq 0$ with $\vartheta(q_1) = \vartheta(q_2) = \un{\ell}$.
By using (C5) and the Fock space representation we have that
\[
k_{q_1, q_2} = p_{q_2} f p_{q_1} = 0
\]
which gives the required contradiction.
We also get that the $\{\psi_p(\K X_p)\}_{\vartheta(p) \in [\un{m}, \un{m} + \un{m}']}$ are linearly independent and so the Fock representation defines a faithful representation of $\N\T(X)$.
\end{proof}

\begin{remark}
The proof of Theorem \ref{T:ex con} also applies to show that $\N\T(X)$ is nuclear when $A$ is nuclear.
As nuclearity passes to quotients we also get that $\N\O(X)$ is nuclear whenever $A$ is so.
This gives an extension of \cite[Corollary 8.3]{Li13} for $\bZ^N$-controlled quasi-lattices.
Indeed $\ca_r(P)$ of \cite{Li13} is $\N\T(X)$ for the trivial product system with $X_p = \bC$ for all $p \in P$.
\end{remark}

Nevertheless there is a conditional expectation onto a ``smaller fixed point algebra'' of $\N\T(X)$.
The existence of such maps will be the prototype for the quotients we will consider. 

\begin{remark}\label{R:cond exp fock}
Let $(\pi,t)$ be the Fock representation.
It follows by \cite[Proposition 5.4]{KL16} that there is a faithful conditional expectation 
\[
E \colon \N\T(X) \to \ol{\spn}\{\psi_p(\K X_p) \mid p \in P\}.
\]
The main argument in our case is the following.
Fix $\un{n} \in \bZ_+^N$ and for each $p \in \vartheta^{-1}(\un{n})$ let the approximate identity $(k_{p,\la})_\la \subseteq \K X_p$.
By passing to a sub-net let the projections $e_p := \text{w*-}\lim_\la \psi_p(k_{p, \la})$, and define
\[
E_p \colon B_{[\un{n}, \un{n}]} \to \psi_p(\K X_p) : f \mapsto e_p f e_p.
\]
That is $E_p$ isolates the $p$-th diagonal entry of $B_{[\un{n}, \un{n}]}$.
Now we can set
\[
E_{\un{n}} := \sumoplus_{p \in \vartheta^{-1}(\un{n})} E_p
\qand
E := \sumoplus_{\un{n} \in \bZ_+^N} E_{\un{n}}.
\]
The proof of Theorem \ref{T:ex con} shows that $\N\T(X)^\be$ is the direct sum of the $B_{[\un{n}, \un{n}]}$ and so $E$ is well defined on $\N\T(X)^\be$.
By composing with the faithful conditional expectation induced by $\{\be_{\un{z}}\}_{\un{z} \in \bT^N}$ we may extend $E$ to the whole $\N\T(X)$.
It follows that $E$ is the projection to the diagonal part of the cores as required, and thus it is faithful.
Likewise we have that $E \otimes \id_D$ and $E \otimes_{\max} \id_D$ is a faithful conditional expectation for every C*-algebra $D$.
\end{remark}

We will also require that $(G,P)$ satisfies a $\bZ_+^N$-controlled minimality condition.
In Section \ref{S:exa} we will see a range of $(G,P)$ that satisfy these conditions.

\begin{definition}\label{D:minimality}
Let $\vartheta \colon (G,P) \to (\bZ^N, \bZ_+^N)$ be a controlled map.
We say that $\vartheta$ has the \emph{minimality property} if, whenever $p \neq q$ are minimal in $\vartheta^{-1}([\un{m}, \infty])$ for $\un{m} \in \bZ_+^N$, then $p \vee q = \infty$.
\end{definition}

Now we can extend the nuclearity results from $(\bZ^N, \bZ_+^N)$ to such $(G,P)$ and for quotients of $\N\T(X)$.
In Proposition \ref{P:CNP is in} we will show that the Cuntz-Nica-Pimsner algebra is included in Theorem \ref{T:nuc con}.

\begin{theorem}\label{T:nuc con}
Let $\vartheta \colon (G,P) \to (\bZ^N, \bZ_+^N)$ be a controlled map with the minimality property and $X$ be a compactly aligned product system over $(G,P)$.
Let $\N\P(X)$ be a $\bT^N$-equivariant quotient of $\N\T(X)$ such that $A \hookrightarrow \N\P(X)$ isometrically and there is a faithful conditional expectation
\[
E \colon \N\P(X) \to \ol{\spn}\{\psi_p(\K X_p) \mid p \in P\}
\]
such that $E \otimes_{\max} \id_D$ (resp. $E \otimes \id_D$) is a faithful conditional expectation of $\N\P(X) \otimes_{\max} D$ (resp. $\N\P(X) \otimes D$) for every C*-algebra $D$.
Set $I := A \cap B_{(\un{0}, \infty]}$.
Then the following are equivalent:
\begin{enumerate}
\item
$A / I$ is nuclear and the embedding $I \hookrightarrow {B}_{(\un{0},\infty]}$ is nuclear.
\item
The embedding $A \hookrightarrow E(\N\P(X))$ is nuclear;
\item
The embedding $A \hookrightarrow \N\P(X)^\be$ is nuclear;
\item
The embedding $A \hookrightarrow \N\P(X)$ is nuclear;
\item
$E(\N\P(X))$ is nuclear;
\item
$\N\P(X)^\be$ is nuclear;
\item
$\N\P(X)$ is nuclear.
\end{enumerate}
\end{theorem}

\begin{proof}
Since $E$ is a conditional expectation and $A \subseteq E(\N\P(X)) \subseteq \N\P(X)^\be$ we get that item (ii) is equivalent to item (iii).
Furthermore the proof of \cite[Theorem 4.5.2]{BO08} implies that item (v) is equivalent to item (vi).
Notice that the argument therein requires just the existence of a faithful conditional expectation $E$ so that $E \otimes_{\max} \id_D$ and $E \otimes \id_D$ are also faithful.
Hence the proof reads the same with that of Theorem \ref{T:nuc NO} as long as we check how the induction works so that item (ii) implies item (v).
This will be done in a number of steps.

For the base step, nuclearity of $A \hookrightarrow E(B_{[\un{0}, \infty]})$ gives nuclearity of the orthogonal embeddings $\psi_p(\K X_p) \hookrightarrow E(B_{[\vartheta(p), \infty]})$.
Hence the embedding $E(B_{[\un{m}, \un{m}]}) \hookrightarrow E(B_{[\un{m}, \infty]})$ is nuclear.
Fix $\un{m}, \un{m}' \in \bZ_+^N$ such that $|\supp \un{m}'| = d+1$.
As in the proof of Theorem \ref{T:nuc NO} let $i \in \supp \un{m}'$ and write
\[
B_k 
:= E(B_{[\un{m} + k \Bi,\un{m} + k\Bi + (\un{m}'- m_i' \Bi)]})
= \ol{\spn} \{\psi_p(\K X_p) \mid \vartheta(p) \in [\un{m} + k \Bi,\un{m} + k\Bi + (\un{m}'- m_i' \Bi)] \},
\]
and likewise
\[
A_n := \sum_{k=n}^{m_i'} B_k
\qand A_n^\infty:= {B}_{[\un{m} + n\Bi, \infty]}
\qfor 
0 \leq n \leq m_i' - 1.
\]
Suppose that the embeddings $B_n \hookrightarrow A_n^\infty$ and $A_{n+1} \hookrightarrow A_{n+1}^\infty$ are nuclear.
We shall show that the embedding $A_n \hookrightarrow A_n^\infty$ is nuclear.

We need to make some preliminary comments that use minimality.
Notice that if $\un{m}' \leq \un{m}'' \leq \infty$ and an element is minimal in $\vartheta^{-1}([\un{m}, \un{m} + \un{m}'])$ then it is also minimal in $\vartheta^{-1}([\un{m}, \un{m} + \un{m}''])$.
The minimality property suggests that if $p \in \vartheta^{-1}([\un{m}, \infty])$ then there is a \emph{unique} element $r$ that is minimal in $\vartheta^{-1}([\un{m}, \infty])$ such that $p \in r P$.
If the minimal $r$ is not in $\vartheta^{-1}([\un{m}, \un{m} + \un{m}'])$ then $r \vee q = \infty$ for every $q \in \vartheta^{-1}([\un{m}, \un{m} + \un{m}'])$.
Next we set the following sets for $n = 0, \dots, m_i'$, based on $\un{m}, \un{m}' \in \bZ_+^N$ and the fixed $i \in \supp \un{m}'$:
\begin{enumerate}
\item $F_{n} := \{r \in P \mid r \text{ is minimal in } \vartheta^{-1}([\un{m} + n \Bi, \un{m} + \un{m}'])\}$,
\item $F_{n + 1} := \{q \in P \mid q \text{ is minimal in } \vartheta^{-1}([\un{m} + (n+1) \Bi, \un{m} + \un{m}'])\}$,
\item $G_{n} := \{s \in P \mid s \text{ is minimal in } \vartheta^{-1}([\un{m} + n \Bi, \un{m} + n \Bi + (\un{m}' - m_i' \Bi)])\}$.
\end{enumerate}
It follows that $G_n \subseteq F_n$.
The schematic of these sets is given in the following figure.

\begin{center}
\begin{tikzpicture}[domain=0:10]
\draw[->] (-8,0) -- (-2,0);

\draw[->] (0,0) -- (6,0);

\draw[color=black, thick] (-7,2.5) -- (-3,2.5);  
\draw[color=black, thick] (2,2.5) -- (5,2.5);  
\draw[color=black, thick] (-7,0) -- (-7,2.5); 

\draw[color=black, thick] (-3,0) -- (-3,2.5);  
\draw[color=black, thick] (1,0) -- (1,2.5); 
\draw[color=black, thick] (5,0) -- (5,2.5);  
\draw[color=black, thick] (2,0) -- (2,2.5);  

\draw[dashed, color=black] (3,0) -- (3,2.5);  
\draw[dashed, color=black] (4,0) -- (4,2.5);  

\draw[dashed, color=black] (-6,0) -- (-6,2.5);  
\draw[dashed, color=black] (-5,0) -- (-5,1);  
\draw[dashed, color=black] (-5,1.55) -- (-5,2.5);  
\draw[dashed, color=black] (-4,0) -- (-4,2.5);  
\draw[dashed, color=black] (2,-.75) -- (2,0);  
\draw[dashed, color=black] (2,2.5) -- (2,3.45);  

\node (a) at (-5,1.25) {$F_{n}$};
\node (b) at (3.5,1.25) {$F_{n+1}$};
\node (c) at (.6,1.25) {$G_n$};
\node (d) at (-7,-.3) {$\un{m} + n \Bi$};
\node (e) at (-3,-.3) {{\small $\un{m} + m_i' \Bi$}};
\node (f) at (1,-.3) {$\un{m} + n \Bi$};
\node (g) at (2, -1) {{\small $\un{m} + (n+1) \Bi$}};
\node (h) at (5,-.3) {{\small $\un{m} + m_i' \Bi$}};
\node (i) at (-7,2.9) {{\small $\un{m} + n \Bi +(\un{m}' - m_i' \Bi)$}};
\node (i) at (0.4,2.9) {{\small $\un{m} + n \Bi +(\un{m}' - m_i' \Bi)$}};
\node (i) at (1.8,3.6) {{\small $\un{m} + (n+1) \Bi +(\un{m}' - m_i' \Bi)$}};
\node (k) at (-2.8,2.9) {{\small $\un{m} + \un{m}'$}};
\node (l) at (5.2,2.9) {{\small $\un{m} + \un{m}'$}};
\end{tikzpicture}

Figure for $F_n$, $G_n$ and $F_{n+1}$.
\end{center}

\smallskip

\noindent
The minimality property shows that $p \vee p' = \infty$ whenever $p, p'$ are distinct elements in each of the above sets, in which case Nica-covariance implies that $\psi_{p}(k_{p}) \psi_{p'}(k_{p'}) = 0$.
Let $(k_{p, \la})_{\la \in \La_p}$ denote an approximate identity of $\K X_{p}$.
Set $\La$ be the directed product of the $\La_r$ for $r \in F_n$ (resp. for $F_{n+1}$) and set the projections that appear as weak*-limits by 
\[
f_{n} := \text{w*-}\lim_{\la \in \La} \sumoplus_{r \in F_n} \psi_{r}(k_{r, \la}), \;
f_{n + 1} := \text{w*-}\lim_{\la \in \La} \sumoplus_{q \in F_{n+1}} \psi_{q}(k_{q, \la})
\; \text{and} \;
g_{n} := \text{w*-}\lim_{\la \in \La} \sumoplus_{s \in G_n} \psi_{s}(k_{s, \la}).
\]
We define the ``corners''
\[
C_n^\infty := f_n A_n^\infty f_n
\qand
C_{n+1}^\infty := f_{n+1} A_{n+1}^\infty f_{n+1}.
\]

\smallskip

\noindent
{\bf Claim 1.} With the aforementioned notation we have that
\[
C_n^\infty = f_n A_n^\infty = \ol{\spn} \{\psi_p(\K X_p) \mid p \in F_n \cdot P\}
\]
and
\[
C_{n+1}^\infty = f_{n+1} A_{n+1}^\infty = \ol{\spn} \{\psi_p(\K X_p) \mid p \in F_{n+1} \cdot P\}.
\]
Moreover we have that
\[
g_n A_n^\infty g_n
=
g_n A_n^{\infty}
=
\ol{\spn}\{ \psi_p( \K X_p) \mid p \in G_n \cdot P \} 
\subseteq
C_n^\infty,
\]
and that
\[
f_n a = a \; \foral a \in A_n, \;
f_{n+1} b = b \; \foral b \in A_{n+1} 
\; \text{ and } \;
g_n c = c \; \foral c \in B_n.
\]

\smallskip

\noindent
{\bf Proof of Claim.}
We begin with $C_n^\infty$.
It suffices to show that $C_n^\infty$ attains the claimed linear span form, and symmetry in the small fixed point algebra will give that $f_n A_n^\infty = A_n^\infty f_n = f_n A_n^\infty f_n$.
To this end let $p \in \vartheta^{-1}([\un{m} + n \Bi, \infty])$ and we have two cases.

\noindent
-- Case 1. 
Suppose that $p \in F_n \cdot P$.
Then there is a unique minimal element $r' \in F_n$ with $p = r' p'$ and $p \vee r = \infty$ for all $r \in F_n$ with $r \neq r'$.
Therefore for $k_p \in \K X_p$ we get that
\begin{align*}
f_n \psi_p(k_p) 
& = 
\text{w*-}\lim_\la \sumoplus_{r \in F_n} \psi_r(k_{r, \la}) \psi_p(k_p) 
 = 
\text{w*-}\lim_\la \psi_{r'}(k_{r', \la}) \psi_{r'p'}(k_{r'p'})
=
\psi(k_p).
\end{align*}

\noindent
-- Case 2.
Suppose that $p \notin F_n \cdot P$ and let $r'$ be the minimal element in $\vartheta^{-1}([\un{m} + n \Bi, \infty])$ so that $p \in r' P$.
Then $r' \vee r = \infty$ and so $p \vee r = \infty$ for all $r \in F_n$, and Nica-covariance yields $f_n \psi_p(\K X_p) = 0$.

This shows that $C_n^\infty$ is the closed linear span of the $\psi_p(\K X_p)$ for $p \in F_{n} \cdot P$.
The same reasoning applies for $C_{n+1}^\infty$.
Applying the same argument on $g_n$ gives also that
\begin{align*}
g_n A_n^\infty g_n
& =
\ol{\spn}\{ \psi_p( \K X_p) \mid p \in G_n \cdot P \} 
 \subseteq
\ol{\spn}\{ \psi_p( \K X_p) \mid p \in F_n \cdot P \} 
=
C_n^\infty.
\end{align*}

For the case of $B_n$, $A_n$ and $A_{n+1}$ we proceed likewise keeping in mind that for every $k_p$ appearing in each of these sets, the index $p$ is larger than a unique minimal element in the corresponding set.
The proof of the claim is complete.
\hfill{$\Box$}

\smallskip

\noindent
{\bf Claim 2.}
The embeddings $B_n \hookrightarrow C_n^\infty$ and $A_{n+1} \hookrightarrow C_{n+1}^\infty$ are nuclear.

\smallskip

\noindent
{\bf Proof of the Claim.}
By using the approximate identities and Claim 1 we have that the mapping
\[
B_n = g_n B_n g_n \hookrightarrow g_n A_n^\infty g_n \subseteq C_n^\infty
\]
is nuclear.
In a similar way we have that the embedding
\[
A_{n+1} = f_{n+1} A_{n+1} f_{n+1} \hookrightarrow f_{n+1} A_{n+1}^\infty f_{n+1} = C_{n+1}^\infty
\]
is nuclear, and the proof of the claim is complete.
\hfill{$\Box$}
\smallskip

\noindent
{\bf Claim 3.}
The C*-algebra $C_{n+1}^\infty$ is an ideal in $C_n^\infty$.

\smallskip

\noindent
{\bf Proof of the Claim.}
First we show that $C_{n+1}^\infty \subseteq C_n^\infty$.
To this end let $\psi_p(k_p) \in C_{n+1}^\infty$ with $p  = q p'$ for $q \in F_{n+1}$.
Let $r \in F_n$ be the unique element with $q = rq'$.
But then $p = r q' p' \in F_n \cdot P$ and so $\psi_p(k_p) \in C_n^\infty$ as well.

Next we show that $C_{n+1}^\infty$ is an ideal in $C_n^\infty$.
For that we need to show that $f_{n+1} f_n = f_{n+1}$ as then we will have that
\begin{align*}
C_{n+1}^\infty \cdot C_n^\infty
& =
f_{n+1} A_{n+1}^\infty A_n^\infty f_n
\subseteq
f_{n+1} A_{n+1}^\infty f_n \\
& =
f_{n+1} A_{n+1}^\infty f_{n+1} f_n
=
f_{n+1} A_{n+1}^{\infty} f_{n+1}
=
C_{n+1}^\infty.
\end{align*}
To this end, by definition we have that
\[
f_{n+1} f_n
=
\text{w*-}\lim_\mu \text{w*-}\lim_\la \sumoplus_{q \in F_{n+1}} \psi_q(k_{q,\mu}) \sumoplus_{r \in F_n} \psi_r(k_{r,\la}). 
\]
Take a $q \in F_{n+1}$.
If $q \in F_n$ as well, then $q \vee r = \infty$ for any $r \in F_n$ with $r \neq q$ and so
\[
\text{w*-}\lim_\mu \text{w*-}\lim_\la \psi_q(k_{q,\mu}) \sumoplus_{r \in F_n} \psi_r(k_{r,\la})
=
\text{w*-}\lim_\mu \text{w*-}\lim_\la \psi_q(k_{q,\mu}) \psi_q(k_{q,\la})
=
\text{w*-}\lim_\mu \psi_q(k_{q,\mu}),
\]
as the last element is a projection.
If $q \notin F_n$ then consider $r' \in F_n$ be the unique minimal element so that $q = r' q'$.
In this case we have that $q \vee r = \infty$ for any $r \in F_n$ with $r \neq r'$ and so
\[
\text{w*-}\lim_\la \psi_q(k_{q,\mu}) \sumoplus_{r \in F_n} \psi_r(k_{r,\la})
=
\text{w*-}\lim_\la \psi_q(k_{q,\mu}) \psi_{r'}(k_{r', \la})
=
\psi_q(k_{q,\mu}).
\]
As this holds for all $q \in F_{n+1}$ we derive that
\[
f_{n+1} \cdot f_n
=
\text{w*-}\lim_\mu \text{w*-}\lim_\la \sumoplus_{q \in F_{n+1}} \psi_q(k_{q,\mu}) \sumoplus_{r \in F_n} \psi_r(k_{r,\la})
=
\text{w*-}\lim_\mu \sumoplus_{q \in F_{n+1}} \psi_q(k_{q,\mu})
=
f_{n+1},
\]
and the proof of the claim is complete.
\hfill{$\Box$}

\smallskip

Now we check that the approximate identity hypothesis of Proposition \ref{P:nuc-ext} is satisfied for $A_n$ and $C_{n+1}^\infty$.
Let $(e_\mu)_\mu$ be the approximate identity of $C_{n+1}^\infty$ with
\[
e_\mu:= \sumoplus_{q \in F_{n+1}} \psi_q(k_{q,\mu}).
\]
However for $\psi_p(k_p) \in B_n$ we have that $\vartheta(p) \in [\un{m} + n \Bi, \un{m} + n \Bi + (\un{m}' - m_i' \Bi)]$, and so for $q \in F_{n+1}$ with $q \vee p < \infty$ we get
\[
\vartheta(q \vee p) = \vartheta(q) \vee \vartheta(p) \in [\un{m} + (n+1) \Bi, \un{m} + \un{m}'].
\]
Thus $e_\mu B_n \subseteq A_{n+1}$.
As $e_\mu$ is also an approximate identity in and for $A_{n+1}$ we have that
\[
e_\mu A_n = e_\mu B_n + e_\mu A_{n+1} \subseteq A_{n+1}.
\]

Now we are set to finish the proof of the inductive step.
That is, first we have observed in Claim 2 that the middle arrow in the following diagram
\begin{align*}
\xymatrix{
0 \ar[r] & A_{n+1} \cap B_n \ar[r] \ar[d] & B_n \ar[r] \ar[d] & B_n / A_{n+1} \cap B_n \ar[r] \ar[d] & 0\\
0 \ar[r] & C_{n+1}^{\infty} \ar[r] & C_n^{\infty} \ar[r] & C_n^{\infty} / C_{n+1}^{\infty} \ar[r] & 0
}
\end{align*}
is nuclear and so the map
\[
A_{n}/A_{n+1} \simeq B_n / (A_{n+1} \cap B_n) \to C_n^\infty / C_{n+1}^\infty
\]
is nuclear by Proposition \ref{P:nuc-quot}.
Moreover we deduced that the map $A_{n+1} \hookrightarrow C_{n+1}^\infty$ is nuclear and that the hypotheses of Proposition \ref{P:nuc-ext} are satisfied.
Hence applying this proposition to the diagram
\begin{equation*}
\xymatrix{
0 \ar[r] & A_{n+1} \ar[r] \ar[d] & A_n \ar[r] \ar[d] & A_n / A_{n+1} \ar[r] \ar[d] & 0\\
0 \ar[r] & C_{n+1}^{\infty} \ar[r] & C_n^{\infty} \ar[r] & C_n^{\infty} / C_{n+1}^{\infty} \ar[r] & 0
}
\end{equation*}
yields that the middle vertical embedding is nuclear.
As $C_n^\infty \subseteq A_n^\infty$ the map $A_n \hookrightarrow A_n^\infty$ is nuclear, and the proof is complete.
\end{proof}

\begin{proposition}\label{P:CNP is in}
Let $\vartheta \colon (G,P) \to (\bZ^N, \bZ_+^N)$ be a controlled map and $X$ be a compactly aligned product system over $(G,P)$.
If $\N\P(X)$ is a CNP-relative quotient of $\N\T(X)$ then it satisfies the hypotheses of Theorem \ref{T:nuc con}.
\end{proposition}

\begin{proof}
First we show that $X$ is $\wt\phi$-injective and so $A$ embeds isometrically in $\N\O(X)$.
As $A \to \N\P(X)$ factors through $A \to \N\O(X)$ we have that $A$ embeds isometrically in $\N\P(X)$.
By \cite[Lemma 3.15]{SY11} we have to show that if $S \subseteq P$ is bounded then it has a maximal element.
To reach a contradiction let $p \in P$ such that $s \leq p$ for all $s \in S$ and suppose there exists a sequence $s_1 < s_2 < \cdots \leq p$ of distinct elements in $S$.
As $\vartheta$ is order-preserving, we have that $\vartheta(S)$ is bounded by $\vartheta(p)$ in $\bZ_+^N$, and thus finite.
However, since all $s_n$ are comparable, by (C4) and (C5) we have that $\vartheta(s_1) < \vartheta(s_2) < \cdots$ which contradicts finiteness of $\vartheta(S)$.

Next we show the existence of the faithful conditional expectation.
Suppose that $\N\P(X) = \ca(\wt\pi, \wt{t})$ and notice that it inherits the $\bT^N$-action from $\N\T(X)$. Also fix $U \colon G \to \B(\ell^2(G))$ be the left regular unitary representation of $G$.
As in the proof of \cite[Lemma 3.1]{DK18a} the family $(\wt{\pi} \otimes I, \wt{t} \otimes  U)$ with
\[
(\wt{t} \otimes U)(\xi_p) = \wt{t}_p(\xi_p) \otimes U_p
\foral \xi_p \in X_p,
\]
defines a representation 
\[
\Phi \colon \N\P(X) \to \N\P(X) \otimes \ca_r(G).
\]

We show that $\Phi$ is injective.
Let $E' \colon \N\P(X) \to \N\P(X)^\be$ be the faithful conditional expectation through the inherited action $\{\be_{\un{z}} \}_{\un{z} \in \bT^N}$ on $\N\P(X)$.
By using the action $\be \otimes \id_{\B(\ell^2(G))}$ we have that the map
\[
E' \otimes \id_{\B(\ell^2(G))} \colon \N\P(X) \otimes \B(\ell^2(G)) \to \N\P(X)^\be \otimes \B(\ell^2(G))
\]
is a faithful conditional expectation.
Thus $\Phi$ is injective if and only if it is injective on the $\bT^N$-fixed point algebra; equivalently, if it is injective on the cores $B_{[\un{0}, \un{m}]}$.
Write $(\pi, t, \psi) = (\Phi \wt\pi, \Phi \wt{t}, \Phi\wt{\psi})$ and recall that $k_{r,s} \in \K(X_s, X_r)$ so that $\wt\psi_{r,s}(k_{r,s}) \in \ol{\spn}\{\wt{t}(X_s) \wt{t}(X_r)^*\}$.
Note that
\begin{equation}\label{eq:inj}
\Phi(\wt{f}) =\wt{f} \otimes I
\foral
\wt{f} \in \spn\{\wt\psi_p(\K X_p) \mid \vartheta(p) \in [\un{0}, \un{m}]\},
\end{equation}
and so the restriction of $\Phi$ on these C*-subalgebras is injective.

\smallskip

\noindent
{\bf Claim.} Let $p, q$ be distinct with $\vartheta(p) = \un{n} = \vartheta(q)$.
If $\vartheta(r) = \vartheta(s) = \un{n}$ with $r \neq p$ or $s \neq q$, or if $r =s$, then
\[
t(X_q)^* \psi_{r,s}(k_{r,s}) t(X_p) = (0).
\]
Consequently if
\[
\sum \{ \psi_{r,s}(k_{r,s}) \mid \vartheta(r) = \vartheta(s) \in [\un{0}, \un{m}]\} = 0
\]
and $\un{0} \neq \un{n}$ is minimal in $[\un{0}, \un{m}]$ for which there exists $k_{p,q} \neq 0$ with distinct $p, q$ in $\vartheta^{-1}(\un{n})$ then
\[
t(X_q)^* \psi_{p,q}(k_{p,q}) t(X_p) \subseteq B_{(\un{0}, \un{m} - \un{n}]}.
\]

\smallskip

\noindent
{\bf Proof of the Claim.}
By (C5) we have that $p \vee q = \infty$.
By Nica-covariance the product is zero whenever the following are not simultaneously satisfied:
\begin{equation}\label{eq:unions}
p \vee r < \infty, q \vee s <\infty \text{ and } r^{-1}(p \vee r) \vee s^{-1}(q \vee s) < \infty.
\end{equation}
By (C5), if $r \neq p$ or $s \neq q$ and $\vartheta(r) = \un{n} =\vartheta(s)$ then the first or the second union does not exist.
Now let $r = s$ and suppose that the unions above do exist, i.e., 
there are $p', q', x, x', y, y' \in P$ such that $p p' = rx$, $q q' = ry$ and $xx' = yy'$.
Then we would have the contradiction $p \vee q < \infty$ as
\[
p p'x' = r xx' = ryy' = q q'y.
\]
Hence the unions in equation (\ref{eq:unions}) are not satisfied and the product is zero when $r = s$.

For the second part, let $\psi_{r,s}(k_{r,s})$ with $\vartheta(r) = \vartheta(s) \in [\un{0}, \un{m}]$.
If all three unions in equation (\ref{eq:unions}) are satisfied, then Nica-covariance yields
\[
t(\xi_q)^* \psi_{r,s}(k_{r,s}) t(\xi_p) \in \psi_{p^{-1} r w, q^{-1} s w}(\K(X_{p^{-1} r w}, X_{q^{-1} s w}))
\qfor
w := r^{-1}(p \vee r) \vee s^{-1}(q \vee s),
\]
otherwise the product is zero.
By the first part the product is zero when $r = s$.
For $r \neq s$ we see that if $w$ exists then
\[
\un{0} \leq - \vartheta(p) + \vartheta(p \vee r) \leq \vartheta(p^{-1} r w) \leq -\vartheta(p) + \vartheta(p) \vee \vartheta(r) \vee \vartheta(q) \vee \vartheta(s) \leq \un{m} - \un{n}.
\]
If $\vartheta(p^{-1} r w) = \un{0} = \vartheta(q^{-1} s w)$ then (C5) yields $p = p \vee r$.
Likewise $q = q \vee s$ and so $r \leq p$ and $s \leq q$.
Minimality of $\un{n}$ forces that either $(p,q) = (r,s)$ or that $k_{r,s} = 0$.
Thus for every $\xi_p \in X_p$ and $\xi_q \in X_q$ there are suitable $k'_{r',s'}$ with $\vartheta(r') = \vartheta(s') \in (\un{0}, \un{m} - \un{n}]$ so that
\begin{align*}
0 
& = \sum \{ t(\xi_q)^* \psi_{r,s}(k_{r,s})  t(\xi_p) \mid \vartheta(r) = \vartheta(s) \in [\un{0}, \un{m}]\} \\
& = t(\xi_q)^* \psi_{p,q}(k_{p,q}) t(\xi_p) + \sum \{ \psi_{r',s'}(k_{r',s'}') \mid \vartheta(r') = \vartheta(s') \in (\un{0}, \un{m} - \un{n}] \},
\end{align*}
and the proof of the claim is complete.
\hfill{$\Box$}
\smallskip

\smallskip

For the base step first we show that $\Phi$ is injective on every $\wt{B}_{[\un{m}, \un{m}]}$.
To this end let
\[
\wt{f} := \sum\{ \wt\psi_{r,s}(k_{r,s}) \mid \vartheta(r) = \vartheta(s) = \un{m}\} \in \ker\Phi.
\]
We have that $p \vee r = q \vee s = \infty$ for all $(r,s) \neq (p, q)$ in the sum, and so
\[
\psi_q(\K X_q) \psi_{p,q}(k_{p,q}) \psi_p(\K X_p) =
\psi_q(\K X_q) \Phi(\wt{f}) \psi_p(\K X_p) = 
(0).
\]
By using an approximate identity we have that $\psi_{p,q}(k_{p,q}) = 0$, and so $k_{p,q} = 0$ for all $(p,q)$ in the sum, giving that $\wt{f} = 0$.
Secondly we show that $\Phi$ is injective on $\wt{B}_{[\un{0}, \Bi]}$.
To this end let
\[
\wt{f} := \sum\{ \wt\psi_{r,s}(k_{r,s}) \mid \vartheta(r) = \vartheta(s) \in [\un{0}, \Bi]\} \in \ker\Phi.
\]
Let $p \neq q$ in $\vartheta^{-1}([\un{0}, \Bi])$; then by (C3) we have that $\vartheta(p) = \vartheta(q) = \Bi$.
By the Claim we derive that
\[
\psi_q(\K X_q) \psi_{p,q}(k_{p,q}) \psi_p(\K X_p) = \psi_q(\K X_q) \Phi(\wt{f}) \psi_p(\K X_p) = (0).
\]
An approximate identity argument shows that $k_{p,q} = 0$ whenever $p \neq q$.
Therefore
\[
\wt{f} = \sum\{ \wt\psi_{r}(k_{r}) \mid \vartheta(r) \in [\un{0}, \Bi]\},
\]
and so $\wt{f} = 0$ by equation (\ref{eq:inj}).
Next we show that $\Phi$ is injective on $\wt{B}_{[\un{n}, \un{n} + \Bi]}$.
To this end let
\[
\wt{f} := \sum\{ \wt\psi_{r,s}(k_{r,s}) \mid \vartheta(r) = \vartheta(s) \in [\un{n}, \un{n} + \Bi]\} \in \ker\Phi,
\]
with the understanding that for all $k_{r,s}$ in the sum with $r \neq s$ we have that either $\wt{\psi}_{r,s}(k_{r,s}) = 0$ or that $\wt\psi_{r,s}(k_{r,s}) \notin \wt{B}_{(\vartheta(r), \un{n} + \Bi]}$.
Suppose that there are $p \neq q$ such that $\wt\psi_{p,q}(k_{p,q}) \neq 0$ and assume first that this happens for $p, q \in \vartheta^{-1}(\un{n})$.
Then by the Claim we derive that
\[
\pi(A) \ni t(X_q)^* \psi_{p,q}(k_{p,q}) t(X_p) \subseteq B_{[\Bi, \Bi]}.
\]
As $\Phi$ is injective on $\wt{B}_{[\un{0}, \Bi]}$ and by using an approximate identity argument we get $\wt\psi_{r,s}(k_{r,s}) \in \wt{B}_{(\vartheta(r), \un{n} + \Bi]}$, which is a contradiction.
Hence $k_{r,s} = 0$ for all $r \neq s$ with $\vartheta(r) = \vartheta(s) = \un{n}$.
Now we can proceed as in the base case to show that $k_{r,s} = 0$ for all distinct $r, s$ in $\vartheta^{-1}(\un{n} + \Bi)$ as well, and get that
\[
\wt{f} = \sum\{ \wt\psi_{r}(k_{r}) \mid \vartheta(r) \in [\un{n}, \un{n} + \Bi]\}.
\]
But then $\wt{f} = 0$ by equation (\ref{eq:inj}).
This finishes the proof of the base step.

For the inductive step suppose that $\Phi$ is injective on all proper subsets $[\un{n}, \un{n}']$ of $[\un{0}, \un{m}]$ and we will show that it is injective on $\wt{B}_{[\un{0}, \un{m}]}$.
To this end let
\[
\wt{f} := \sum\{ \wt\psi_{r,s}(k_{r,s}) \mid \vartheta(r) = \vartheta(s) \in [\un{0}, \un{m}]\} \in \ker\Phi,
\]
with the understanding that for all $k_{r,s}$ in the sum with $r \neq s$ we have that either $\wt{\psi}_{r,s}(k_{r,s}) = 0$ or that $\wt\psi_{r,s}(k_{r,s}) \notin \wt{B}_{(\vartheta(r), \un{m}]}$.
Let $\un{0} \neq \un{n}$ be minimal so that $\wt\psi_{p,q}(k_{p,q}) \notin \wt{B}_{(\vartheta(p), \un{m}]}$ with $p \neq q$ and $\vartheta(p) = \un{n} =\vartheta(q)$.
Then by the Claim we derive that
\[
\pi(A) \ni t(X_q)^* \psi_{p,q}(k_{p,q}) t(X_p) \subseteq B_{(\un{0}, \un{m} - \un{n}]}.
\]
As $\un{m} - \un{n} < \un{m}$, injectivity of $\Phi$ on $\wt{B}_{[\un{0}, \un{m} - \un{n}]}$ and an approximate identity argument yields that $\wt\psi_{p,q}(k_{p,q}) \in \wt{B}_{(\vartheta(p), \un{m}]}$ which is a contradiction.
Hence we have
\[
\wt{f} = \sum\{ \wt\psi_{r}(k_{r}) \mid \vartheta(r) \in [\un{0}, \un{m}]\} \in \ker\Phi,
\]
and so $\wt{f} = 0$ by equation (\ref{eq:inj}).
This finishes the proof that $\Phi$ is injective.

Let $\tau_e \colon \ca_r(G) \to \bC$ be the compression to the $(e,e)$-position.
By \cite[Proposition 4.1.9]{BO08} the map $\id_{\N\P(X)} \otimes \tau_e$ is also a faithful conditional expectation.
Therefore the map
\[
E := (\id_{\N\P(X)} \otimes \tau_e) \circ (E' \otimes \id_{\B(\ell^2(G))}) \circ \Phi  \colon \N\P(X) \to \N\P(X) \otimes \ca_r(G)
\]
is a faithful conditional expectation with range
\begin{align*}
(\id_{\N\P(X)} \otimes \tau_e) \circ (E' \otimes \id) (\ol{\spn}\{t(X_p) t(X_q)^* \otimes U_{pq^{-1}} \mid p, q \in P\}) 
& =\\
& \hspace{-5.5cm} =
(\id_{\N\P(X)} \otimes \tau_e)(\ol{\spn}\{t(X_p) t(X_q)^* \otimes U_{pq^{-1}} \mid \vartheta(p) = \vartheta(q)\}) \\
& \hspace{-5.5cm}  =
\ol{\spn}\{\psi_p(\K X_p) \mid p \in P\}.
\end{align*}
For $\N\P(X) = \N\T(X)$ we recover the conditional expectation of Remark \ref{R:cond exp fock}.

Now fix a C*-algebra $D$.
As before the map $E' \otimes \id_{\B(\ell^2(G))} \otimes \id_D$ is a faithful conditional expectation coming from $\be \otimes \id_{\B(\ell^2(G))} \otimes \id_D$.
An application of \cite[Lemma 4.1.8]{BO08} gives that $\id_{\N\P(X)} \otimes \tau_e \otimes \id_D$ is a faithful conditional expectation, and thus so is the induced $E \otimes \id_D$ on $\N\P(X) \otimes D$.
We want to show that the same holds for $\otimes_{\max}$.
Suppose that $\N\P(X) \otimes_{\max} D$ is faithfully represented on $\B(H)$ by $\Phi_1 \times \Phi_2$ for the commuting representations
\[
\Phi_1 \colon \N\P(X) \to \B(H)
\qand
\Phi_2 \colon D \to \B(H).
\]
Then $\Phi_1 \otimes U$ and $\Phi_2 \otimes I$ also commute and thus by universality of $\otimes_{\max}$ they define a representation 
\[
\Phi \colon \N\P(X) \otimes_{\max} D \to (\N\P(X) \otimes_{\max} D) \otimes \ca_r(G).
\]
Here $(E'\otimes_{\max} \id_D) \otimes \id_{\ca_r(G)}$ is induced by the gauge action $\{(\be_{\un{z}} \otimes_{\max} \id_D) \otimes \id_{\ca_r(G)} \}_{\un{z} \in \bT^N}$ and so is a faithful conditional expectation.
In a similar way as in the first part we can show that $\Phi$ is injective by applying on the C*-algebra $\N\P(X) \otimes_{\max} D$ geverated by the product system $\{\Phi_1(X_p) \Phi_2(D)\}_{p \in P}$.
Now we can apply \cite[Lemma 4.1.8]{BO08} to $(\id_{\N\P(X)} \otimes_{\max} \id_D) \otimes \tau_e = \id_{\N\P(X) \otimes_{\max} D} \otimes \tau_e$ to derive that it is a faithful conditional expectation.
Since
\[
E \otimes_{\max} \id_D = ((\id_{\N\P(X)} \otimes_{\max} \id_D) \otimes \tau_e) \circ ((E'\otimes_{\max} \id_D) \otimes \id_{\ca_r(G)}) \circ \Phi
\]
we conclude that $E \otimes_{\max} \id_D$ is a faithful conditional expectation.
\end{proof}

\section{Examples}\label{S:exa}

Major examples of product systems and their operator algebras come from higher-rank graphs established by Kumjian-Pask \cite{KP00} or from the generalized crossed products of semigroup actions of Fowler \cite{Fow02}.
The former are always nuclear as product systems over $(\bZ^N, \bZ_+^N)$ with an abelian $A$, while the nuclearity of the latter depends on the embedding of $A$ in the fixed point algebra (and not solely on $A$) when the quasi-lattice satisfies the properties of our main theorem.
In this section we give some more examples of quasi-lattice ordered groups (and constructions) that are compatible with controlled maps in $(\bZ^N, \bZ_+^N)$ with the minimality property.
What appears to be essential is a unique normal form for expressing the elements.

\subsection{The Baumslag-Solitar group for $n = m >0$}

Recall that the Baumslag-Solitar group is defined as $B(m,n) = \sca{a, b \mid a^n b = b a^m}$.
Let $B_+(m,n)$ be its sub-semigroup generated by $a, b$.
Spielberg \cite[Theorem 2.11]{Spi12} has shown that $(B(m,n), B_+(m,n))$ is a quasi-lattice ordered group. 

We consider the case where $m=n$.
By Britton's Lemma we can write every $\ol{p} \in B_+(m,m)$ in a unique way as
\[
\ol{p} = a^{p_1} b a^{p_2} b \cdots a^{p_k} b a^{p_{k+1}}
\qfor p_1, \dots, p_k \in \{0, \dots, m-1\}  \qand p_{k+1} \in \bZ_+.
\]
As the commutation relation $ab = ba$ dominates $a^m b = b a^m$ we have that the abelianization gives a surjective controlled map
\[
\vartheta \colon (B(m,m), B_+(m,m)) \to (\bZ^2, \bZ_+^2)
\]
that counts the occurrences of $a$ and $b$ in the unique expression of $\ol{p} = a^{p_1} b a^{p_2} b \cdots a^{p_k} b a^{p_{k+1}}$ so that
\[
\vartheta(\ol{p}) \equiv \vartheta(a^{p_1} b a^{p_2} b \cdots a^{p_k} b a^{p_{k+1}}) : = (p_1 + \cdots + p_{k+1}, k).
\]

\begin{proposition}
The abelinization map $\vartheta \colon (B(m,m), B_+(m,m)) \to (\bZ^2, \bZ_+^2)$ has the $\bZ_+^2$-minimality property.
\end{proposition}

\begin{proof}
Suppose that $\ol{p}, \ol{q}$ are minimal in $\vartheta^{-1}([(m_1, m_2), \infty])$ so that $\ol{p} \vee \ol{q} < \infty$.
Then there are $\ol{r}, \ol{s} \in B_+(m,m)$ such that
\begin{align*}
 a^{p_1} b \cdots a^{p_k} b a^{\upsilon} b a^{r_2} b \cdots b a^{r_{\ell+1} + \pi m}
& =
\ol{p} \cdot \ol{r} \\
& =
\ol{p} \vee \ol{q} \\
& =
\ol{q} \cdot \ol{s}
=
 a^{q_1} b \cdots a^{q_{k'}} b a^{\upsilon'} b a^{s_2} b \cdots b a^{s_{\ell'+1} + \pi' m},
\end{align*}
for $p_{k+1} + r_1 = \pi m + \upsilon$ and $q_{k'+1} + s_1 = \pi' m + \upsilon'$.
By assumption $p_1 + \cdots + p_{k + 1} \geq m_1$, $q_1 + \cdots + q_{k' + 1} \geq m_1$ and $k, k' \geq m_2$.
Uniqueness of the expansion of $\ol{p} \vee \ol{q}$ gives that
\[
p_i
=
q_i \qfor i = 1, \dots, \min\{k, k'\}.
\]
First we show that $k = k'$.
If $k > k'$ then let $\ol{p}' = a^{p_1} b \cdots b a^{p_{k' + 1}}$ for which we get that $\ol{p}' < \ol{p}$.
Moreover
\[
p_1 + \cdots + p_{k'+1}
=
q_1 + \cdots + q_{k'+1}
\geq
m_1
\]
and therefore we deduce that $\vartheta(\ol{p}') = (q_1 + \cdots + q_{k'+1}, k') \geq (m_1, m_2)$.
But this contradicts minimality of $\ol{p}$ in $\vartheta^{-1}([(m_1, m_2), \infty])$.
By symmetry we also have that $k \not< k'$ and thus $k = k'$.
Now if $q_{k+1} < p_{k+1}$ then $\ol{q} < \ol{p}$ which is a contradiction; likewise if $q_{k+1} > p_{k+1}$.
Therefore we get the required $\ol{p} = \ol{q}$.
\end{proof}

\subsection{Semidirect products}

Let $(G,P)$ and $(H,S)$ be quasi-lattice ordered groups and let an action $\al \colon H \to \Aut(G)$ such that $\al|_S \colon S \to \Aut(P)$ restricts to automorphisms of $P$.
Then we can form the semi-direct products $G \rtimes_\al H$ and $P \rtimes_\al S$ with respect to the relations $\al_h(g) h = h g$.

The condition on $\al$ makes $P \cdot S$ a subsemigroup of the semi-direct product, and the pair $(G \rtimes_\al H, P \rtimes_\al S)$ is quasi-lattice ordered.
Indeed, recall that every element $x \in G \rtimes_\al H$ has a form as $x = g h$ for unique $g \in G$ and $h \in H$.
Suppose thus that $g_1 h_1$ and $g_2 h_2$ have a common upper bound in $P \rtimes_\al S$.
Hence there are $r_1, r_2 \in G$ and $t_1, t_2 \in H$ such that
\begin{equation}\label{eq:sd lo}
(g_1 \al_{h_1}(r_1)) (h_1 t_1)
=
(g_1 h_1) (r_1 t_1)
=
(g_2 h_2) (r_2 t_2)
=
(g_2 \al_{h_2}(r_2)) (h_2 t_2).
\end{equation}
Uniqueness of the form shows that $g_1 \vee g_2 < \infty$ in $P$ and $h_1 \vee h_2 < \infty$ in $S$.
Therefore $g_1, g_2$ have a least common upper bound say $p \in P$ and $h_1, h_2$ have a least common upper bound $ s \in S$.
We can apply equation (\ref{eq:sd lo}) in reverse using the $r_1, r_2$ and $t_1, t_2$ that satisfy
\[
g_1 \al_{h_1}(r_1) = p = g_2 \al_{h_2}(r_2)
\qand
h_1 t_1 = s = h_2 t_2,
\]
and we derive that $p s$ is the least common upper bound in $P \rtimes_\al S$.
In particular this shows that $(g_1 h_1) \vee (g_2 h_2)$ if and only if $g_1 \vee g_2 <\infty$ and $h_1 \vee h_2 < \infty$, in which case we get
\begin{equation}\label{eq:sd vee}
(g_1 h_1) \vee (g_2 h_2) = (g_1 \vee g_2) (h_1 \vee h_2).
\end{equation}

Now suppose that $(G, P)$ admits a controlled map $\vartheta_N$ in $(\bZ^{N}, \bZ_+^{N})$ and $(H, S)$ admits a controlled map $\vartheta_M$ in $(\bZ^{M}, \bZ_+^{M})$.
In order for the semi-direct product to inherit the obvious controlled map on $(\bZ^{N+M}, \bZ_+^{N+M})$ it is necessary that $\al$ is $\vartheta_N$-invariant in the sense that $\vartheta_N \al_h = \vartheta_N$ for all $h \in H$.
We can then define the homomorphism
\[
\vartheta \colon (G \rtimes_\al H, P \rtimes_\al S) \to (\bZ^{N+M}, \bZ_+^{N+M}) \textup{ such that } \vartheta(g h) = (\vartheta_N(g), \vartheta_M(h)).
\]

\begin{proposition}
If the maps $\vartheta_N \colon (G, P) \to (\bZ^N, \bZ_+^N)$ and $\vartheta_M \colon (H, S) \to (\bZ^M, \bZ_+^M)$ have the minimality property and $\vartheta_N \al_h = \vartheta_N$ for all $h \in H$, then the induced map
\[
\vartheta \colon (G \rtimes_\al H, P \rtimes_\al S) \to (\bZ^{N + M}, \bZ_+^{N + M})
\]
satisfies the $\bZ_+^{N + M}$-minimality property.
\end{proposition}

\begin{proof}
Let $p_1 h_1$ and $p_2 h_2$ be minimal in $\vartheta^{-1}([(\un{n}, \un{m}), \infty])$ with $(p_1 h_1) \vee (p_2 h_2) < \infty$, for $(\un{n}, \un{m}) \in \bZ_+^N \times \bZ_+^M$.
Then $p_1 \vee p_2 < \infty$ and $h_1 \vee h_2 < \infty$.
It also follows that $p_1, p_2$ are minimal in $\vartheta^{-1}_N([\un{n}, \infty])$ and $h_1, h_2$ are minimal in $\vartheta^{-1}_M([\un{m}, \infty])$.
Indeed if $p_1$ were not minimal in $\vartheta^{-1}_N([\un{n}, \infty])$ then let $r \in \vartheta^{-1}_N([\un{n}, \infty])$ so that $p_1 = r r'$.
This gives the contradiction that $p_1 h_1 = rr'h_1 = r h_1 \al_{h_1}^{-1}(r')$ with
\[
\vartheta(r h_1) \geq (\vartheta_N(r), 0) + (0, \vartheta_M(h_1)) \geq (\un{n}, \un{m}).
\]
Likewise if $h_1$ were not minimal in $\vartheta^{-1}_M([\un{m}, \infty])$ then write $h_1 = t t'$ with $\vartheta_M(t) \geq \un{m}$.
But then we get the contradiction $p_1 h_1 = (p_1 t) t'$ with
\[
\vartheta(p_1 t) \geq (\vartheta_N(p_1), 0) + (0, \vartheta_M(t)) \geq (\un{n}, \un{m}).
\]
As $p_1 \vee p_2 < \infty$, $h_1 \vee h_2 < \infty$ and they are minimal we have that $p_1 = p_2$ and $h_1 = h_2$ giving that $p_1 h_1 = p_2 h_2$.
\end{proof}

\subsection{Graph products}

Let $\Ga$ be an undirected finite graph such that there is at most one vertex between any two vertices, and it has no loops.
Let $(G_I, P_I)$ be a quasi-lattice ordered group for each vertex $I \in \Ga$.
The graph product $G \equiv G_\Ga$ is then the group generated by $\sqcup_{I \in \Ga} G_I$ subject to identifying the units $e_I$ and to allowing commutation of elements between $G_I$ and $G_J$ if and only if the vertices $I$ and $J$ are adjacent.
By letting $P$ be the subgroup generated by the $P_I$ in $G$ one obtains a pair $(G, P)$.
Crisp-Laca \cite{CL02} showed that $(G,P)$ is a quasi-lattice ordered group.
When $(G_I, P_I) = (\bZ, \bZ_+)$ for all $I$ then we get a right-angled Artin group on the graph $\Ga$.
We will show that Theorem \ref{T:nuc con} accommodates such settings.

Green \cite{Gre90} has shown that an element $p \in P$ admits a reduced expression $p = p_{I_1} \cdots p_{I_k}$ such that if $I_i = I_j$ then there exists a $k$ between $i$ and $j$ such that $I_k$ is not adjacent to $I_i$.
A syllable $p_{I_i} \neq e_{I_i}$ is called \emph{initial} if $I_i$ is adjacent to all $I_j$ for $j < i$.
The vertex $I_i$ is called \emph{initial} and we write $\De(p)$ for the set of all initial vertices of $p$.
As reduced expressions are shuffle equivalent, it follows that $\De(p)$ (and thus $p_I$) does not depend on the reduced form of $p$.
For $p \in P$ and $I \in \Ga$ let
\[
p(I)
:=
\begin{cases}
p_I & \text{ if } I \in \De(p), \\
e = e_I & \text{ otherwise}.
\end{cases}
\]
For convenience we will also write
\[
C(p) := \{I \in \Ga \mid I \text{ is adjacent to all vertices on } p\}.
\]
Let $p, q \in P$ and write $p = p(I) p'$ and $q = q(I) q'$ for a fixed $I \in \Ga$.
Crisp-Laca \cite[Definition 12]{CL02} define $p, q \in P$ to be $I$-adjacent if:
\begin{enumerate}
\item $p \vee q < \infty$;
\item either $p(I) = p(I) \vee q(I)$ or $I \in C(p(I)^{-1} p)$;
\item either $q(I) = p(I) \vee q(I)$ or $I \in C(q(I)^{-1} q)$.
\end{enumerate}
Then in \cite[Proposition 13]{CL02} it is shown that $p \vee q < \infty$ if and only if they are $I$-adjacent and $p' \vee q' < \infty$, in which case
\[
p \vee q = (p(I) \vee q(I)) \cdot (p' \vee q').
\]

Passing now to controlled maps, suppose that every $(G_I, P_I)$ obtains a controlled map over $(\bZ_+^{N_I}, \bZ_+^{N_I})$.
By taking the completion of $\Ga$, Crisp-Laca \cite{CL07} show that the graph product $(G, P)$ obtains a controlled map over $(\bZ^N, \bZ_+^N)$ for $N := \sum_I N_I$.
We will write $m_I \in \bZ^{N_I}$ for the $I$-part of $\un{m} \in \bZ^{N}$.
By using the embeddings $\bZ^{N_I} \hookrightarrow \bZ^N$ we may write $\un{m} = \sum_{I \in \Ga} m_{I}$.

\begin{proposition}
Let $(G,P)$ be the graph product of the quasi-lattices $(G_I, P_I)$.
If the maps $\vartheta_I \colon (G_I, P_I) \to (\bZ^{N_I}, \bZ^{N_I}_+)$ have the minimality property then so does the completion map
\[
\vartheta \colon (G, P) \to (\bZ^N, \bZ_+^N) \qfor N := \sum_I N_I.
\]
\end{proposition}

\begin{proof}
We start by noticing that if $p = xy$ is minimal in $\vartheta^{-1}([\un{m}, \infty])$ then $y$ is minimal in $\vartheta^{-1}([(\un{m} - \vartheta(x)) \vee \un{0}, \infty])$.
Thus we can reduce by deleting common left parts.

Let $p, q \in P$ such that $p \vee q < \infty$.
Start with $I \in \De(p) \setminus C(p)$ and apply part (ii) of \cite[Proposition 13]{CL02} to get that $p_I \vee q(I) = p_I \neq e_I$.
If $I \notin C(q(I)^{-1} q)$ then $q(I) = p_I$, otherwise $q(I) \leq p_I$.
Now \cite[Proposition 13]{CL02} asserts that $(p_I^{-1} p) \vee (q(I)^{-1} q) < \infty$.
We continue until we exhaust subsequent initial vertices to obtain $r, s \in P$ with
\[
p = p_{I_1} \cdots p_{I_k} r
\qand
q = q(I_1) \cdots q(I_k) s
\]
so that the vertices $I_1, \dots, I_k \notin C(p)$, the $r$ is on $C(p)$, and with $q(I_i) = p_{I_i}$ when $I_i$ is not in $C(q(I_{i-1})^{-1} \cdots q(I_1)^{-1} q)$ and $q(I_i) \leq p_{I_i}$ when $I$ is in $C(q(I_{i-1})^{-1} \cdots q(I_1)^{-1} q)  \supseteq \{I_i, \dots, I_k\}$.
Therefore
\[
q(I_1) \cdots q(I_k) \cdot x(I_1) \cdots x(I_k) = [q(I_1) x(I_1)] \cdot [q(I_2) x(I_2)] \cdots [q(I_k) x(I_k)] = p_{I_1} \cdots p_{I_k}
\]
for
\[
x(I_i) :=
\begin{cases}
e & \textup{ if } I_i \notin C(q(I_{i-1})^{-1} \cdots q(I_1)^{-1} q), \\
q(I_i)^{-1} p_{I_i} & \textup{ if } I_i \in C(q(I_{i-1})^{-1} \cdots q(I_1)^{-1} q).
\end{cases}
\]
Thus $q(I_1) \cdots q(I_k) \leq p_{I_1} \cdots p_{I_k}$ so that we can delete the smaller part.
Now we repeat for $q$ in place of $p$.
Hence without loss of generality we can write
\[
p = p' p'' \qand q = q' q''
\]
where $p''$ is on $C(p)$ with $p''_I = p_I$ and $q(I) = e$ for all vertices $I$ in $p'$.
Likewise for the vertices for $q'$ and $q''$ and so $\vartheta(p') \wedge \vartheta(q') = \un{0}$.

\smallskip

\noindent
{\bf Case 1.} Suppose first that $p = p''$ and $q = q''$.
We will show that $p = q$.
To this end write 
\[
p = p_{I_1} \cdots p_{I_k} \qand q = q_{J_1} \cdots q_{J_\ell}
\]
such that $I_i \in C(p)$ and $J_j \in C(q)$.
Without loss of generality suppose that $I_1 \notin \{J_1, \dots, J_{\ell}\}$ so that $p_{I_1} > e_{I_1} = q(I_1)$, and write
\[
p = p_{I_1} \cdots p_{I_k} \qand q = e_{I_1} q_{J_1} \cdots q_{J_\ell}.
\]
Clearly $q(I_1) \neq p_{I_1} \vee q(I_1)$ and we may invoke part (iii) of \cite[Proposition 13]{CL02} to deduce that $I_1$ must be adjacent to all $J_1, \dots, J_\ell$.
Continuing inductively we get that $I_i$ and $J_j$ are adjacent when $I_i \neq J_j$.
Therefore we can write
\[
p = p(I_1) \cdots p(I_n) \qand q = q(I_1) \cdots q(I_n)
\]
such that all $I_1, \dots, I_n$ are adjacent, but with the understanding that now some of the elements may be the identity.
By \cite[Proposition 13]{CL02} we have that
\[
p(I_i) \vee q(I_i) < \infty \foral i = 1, \dots, n.
\]
Let us write $\un{m} = \sum_{i=1}^n m_{I_i}$ by using the embeddings $\bZ^{N_I} \hookrightarrow \bZ^N$.
If $p(I_1)$ is not minimal in $\vartheta^{-1}_{I_1}([m_{I_1}, \infty])$ then we can write $p(I_1) = r(I_1) r'(I_1)$ with $\vartheta(r(I_1)) \geq m_{I_1}$.
But then we would have that
\[
p = r(I_1)  r'(I_1) p(I_2) \cdots p(I_n) 
= \left(r(I_1) p(I_2) \cdots p(I_n) \right) \cdot r'(I_1),
\]
as the $I_i$ are all adjacent, and with
\[
\vartheta(r(I_1) p(I_2) \cdots p(I_n))
=
\vartheta(r(I_1)) + \vartheta(p(I_2)) + \cdots + \vartheta(p(I_n))
\geq
m_{I_1} + m_{I_2} + \cdots + m_{I_n}
=
\un{m}.
\]
This contradicts minimality of $p$.
Likewise for $q(I_1)$.
Therefore $p(I_1)$ and $q(I_1)$ are minimal in $\vartheta^{-1}([m_{I_1}, \infty])$.
However as $p(I_1) \vee q(I_1) < \infty$ and $(G_{I_1}, P_{I_1})$ has the minimality property, then we get that $p(I_1) = q(I_1)$.
Since all $I_i$ are adjacent, we may apply for $I_i$ in the place of $I_1$ and thus get that $p(I_i) = q(I_i)$ for all $i = 1, \dots, n$ giving the required $p = q$.

\smallskip

\noindent
{\bf Case 2.} Let us now pass to the general case for $p = p' p''$ and $q = q' q''$ with $p''$ on $C(p)$ and $q''$ on $C(q)$ and $\vartheta(p') \wedge \vartheta(q') = \un{0}$.
Write
\[
p = p'' p' \qand q = q'' q'.
\]
As $p \vee q < \infty$ we get that $p'' \vee q'' < \infty$ and the first part of Case 1 gives that the vertices on $p''$ and $q''$ are adjacent.
We can use that every $I$ on $p''$ is in $C(p)$ and so
\[
\un{m}_I \leq \vartheta(p)_I = \vartheta_I(p_I) = \vartheta_I(p''_I) = \vartheta(p'')_I \foral I \in C(p).
\]
Now proceed as in Case 1 to show that $p''_I$ is minimal in $\vartheta^{-1}([m_I, \infty])$.
Likewise for $q''(I)$ and applying for all vertices in $C(p) \cup C(q)$ gives that $p'' = q'' =: s$.
By deleting this left common part we get that $p'$ and $q'$ are minimal in $\vartheta^{-1}([(\un{m} - \vartheta(s)) \vee \un{0}, \infty])$ with $p' \vee q' < \infty$.
But recall that $\vartheta(p')$ and $\vartheta(q')$ are supported on disjoint directions and so
\[
(\un{m} - \vartheta(s)) \vee \un{0} \leq \vartheta(p') \wedge \vartheta(q') = \un{0}.
\]
Thus $p'$ and $q'$ are minimal in $\vartheta^{-1}([\un{0}, \infty])$ and so $p' = q' = e$.
Consequently $p = p'' = q'' = q$, and the proof is complete.
\end{proof}

\begin{acknow}
The author would like to thank Adam Dor-On for his plenty helpful suggestions and comments on a preprint of this paper.
Part of this research was conducted during the visit of the author at the Lorenz Center for the ``Cuntz-Pimsner Cross-Pollination'' workshop.
The author is grateful to the Conference organizers for the invitation to participate and the warm hospitality.
\end{acknow}



\begin{thebibliography}{99}

\bibitem{AM15} S. Albandik and R. Meyer,
\textit{Product systems over Ore monoids},
Documenta Math.\ \textbf{20} (2015), 1331--1402.

\bibitem{Arv89} W. Arveson, 
\textit{Continuous analogues of Fock space}, 
Memoirs Amer. Math. Soc. {\bf 80} (1989).

\bibitem{BO08} N. Brown and N. Ozawa,
\emph{$\ca$-algebras and finite dimensional approximations},
volume 88 of \textit{Graduate Studies in Mathematics}, American Mathematical Society, Providence RI, 2008.

\bibitem{CLSV11} T.M. Carlsen, N.S. Larsen, A. Sims and S.T. Vittadello,
\textit{Co-universal algebras associated to product systems, and gauge-invariant uniquenss theorems},
Proc.\ London Math.\ Soc.\ (3) \textbf{103} (2011), 563--600.

\bibitem{CL02} J. Crisp and M. Laca, 
\textit{On the Toeplitz algebras of right-angled and finite-type Artin groups}, 
J.\ Austr.\ Math.\ Soc.\ \textbf{72} (2002), 223--245.

\bibitem{CL07} J. Crisp and M. Laca, 
\textit{Boundary quotients of Toeplitz algebras of right-angled Artin groups}, 
J.\ Funct.\ Anal.\ \textbf{242} (2007), 125--156.

\bibitem{Dea07} V. Deaconu,
\textit{Iterating the Pimsner construction},
New York J.\ Math.\ \textbf{13} (2007), 199--213.

\bibitem{DK18b} A. Dor-On and E.T.A. Kakariadis, 
\textit{Operator algebras for higher rank analysis}, 
preprint at arXiv: 1803.11260.

\bibitem{DK18a} A. Dor-On and E.G. Katsoulis, 
\textit{Tensor algebras of product systems and their C*-envelopes}, 
preprint at arXiv: 1801.07296.

\bibitem{Fle17} J. Fletcher, 
\textit{Iterating the Cuntz-Nica-Pimsner construction for compactly aligned product systems},
preprint at arXiv:\-1706.08626.

\bibitem{Fow02} N.J. Fowler, 
\textit{Discrete product systems of Hilbert bimodules}, 
Pacific J.\ Math.\ \textbf{204} (2002), 335--375.

\bibitem{Gre90} E.R. Green, 
\textit{Graph products of groups},
PhD Thesis, The University of Leeds, 1990.

\bibitem{Kak16} E. T.A. Kakariadis,
\textit{A note on the Gauge Invariant Uniqueness Theorem for C*-correspondences},
Israel J. Math. {\bf 215} (2016), 513--521. 

\bibitem{Kak17} E.T.A. Kakariadis,
\textit{On Nica-Pimsner algebras of C*-dynamical systems over $\mathds{Z}_+^n$}, 
Int.\ Math.\ Res.\ Not.\ IMRN {\bf 4} (2017), 1013--1065. 

\bibitem{KP13} E.T.A. Kakariadis and J.R. Peters, 
\textit{Representations of C*-dynamical systems implemented by Cuntz families},  
M\"{u}nster J.\ Math.\ \textbf{6} (2013), 383--411.

\bibitem{Kat04} T. Katsura,
\textit{On C*-algebras associated with C*-correspondences},
J.\ Funct.\ Anal.\ \textbf{217} (2004), 366--401.

\bibitem{KP00} A. Kumjian and D. Pask, 
\textit{Higher rank graph C*-algebras}, 
New York J.\ Math.\ \textbf{6} (2000), 1--20.

\bibitem{KPS15} A. Kumjian, D. Pask and A. Sims, 
\textit{On twisted higher-rank graph C*-algebras}, 
Trans.\ Amer.\ Math.\ Soc.\ \textbf{367} (2015), 5177--5216.

\bibitem{KL16} B.K. Kwa\'{s}niewski and N.S. Larsen, 
\textit{Nica-Toeplitz algebras associated with right tensor C*-precategories over right LCM semigroups: part I uniqueness results}, 
preprint at arXiv: 1611.08525.

\bibitem{Lan95} E.C. Lance,
\textit{Hilbert C*-modules. A toolkit for operator algebraists},
London Mathematical Society Lecture Note Series, \textbf{210}. Cambridge University Press, Cambridge, 1995.

\bibitem{Li13} X. Li, 
\textit{Nuclearity of semigroup C*-algebras and the connection to amenability}, 
Adv.\ Math.\ \textbf{244} (2013), 626--662. 

\bibitem{MS98} P. S. Muhly and B. Solel, 
\textit{Tensor algebras over C*-correspondences: representations, dilations and C*-envelopes},
J.\ Funct.\ Anal.\ \textbf{158} (1998), 389--457.

\bibitem{Nic92} A. Nica,
\textit{C*-algebras generated by isometries and Wiener-Hopf operators},
J. Operator Theory \textbf{27} (1992), 17--52.

\bibitem{Pim97} M.V. Pimsner,
\textit{A class of C*-algebras generalizing both Cuntz-Krieger algebras and crossed products by $\bZ$},
Fields Inst. Commun., \textbf{12}, Amer. Math. Soc., Providence, RI, 1997, 189--212.

\bibitem{RRS15} A. Rennie, D. Robertson and A. Sims, 
\textit{Groupoid Fell bundles for product systems over quasi-lattice ordered groups}, 
Mathematical Proceedings of the Cambridge Philosophical Society, to appear.

\bibitem{Seh18} C.F. Sehnem, 
\textit{On C*-algebras associated to product systems}, 
preprint at arXiv: 1804.10546.

\bibitem{Spi12} J. Spielberg, 
\textit{C*-algebras for categories of paths associated to the Baumslag-Solitar groups}, 
J.\ Lond.\ Math.\ Soc.\ {\bf 86} (2012), 728--754.

\bibitem{SY11} A. Sims and T. Yeend,
\textit{Cuntz-Nica-Pimsner algebras associated to product systems of Hilbert bimodules},
J.\ Operator Theory {\bf 64} (2010), 349--376.
\end{thebibliography}
\end{document}